\documentclass[11pt]{article}

\usepackage[margin=1in]{geometry}
\usepackage[utf8]{inputenc}
\usepackage[T1]{fontenc}
\usepackage{lmodern}
\usepackage{amsmath,amssymb,amsfonts,amsthm,mathtools}
\usepackage{bm}
\usepackage{microtype}
\usepackage{enumitem}
\usepackage[colorlinks=true,linkcolor=blue,citecolor=blue,urlcolor=blue]{hyperref}
\usepackage[capitalize,nameinlink]{cleveref}
\hypersetup{
  pdftitle={Filtered Vortex Stretching and Subgrid Defects for the Three-Dimensional Navier--Stokes Equations},
  pdfauthor={Runlong Yu},
  pdfkeywords={Navier--Stokes equations, vortex stretching, filtered vorticity, commutator stress}
}

\numberwithin{equation}{section}

\theoremstyle{plain}
\newtheorem{theorem}{Theorem}[section]
\newtheorem{proposition}[theorem]{Proposition}
\newtheorem{lemma}[theorem]{Lemma}
\newtheorem{corollary}[theorem]{Corollary}

\theoremstyle{definition}
\newtheorem{definition}[theorem]{Definition}

\theoremstyle{remark}
\newtheorem{remark}[theorem]{Remark}

\allowdisplaybreaks[2]
\crefname{theorem}{Theorem}{Theorems}
\Crefname{theorem}{Theorem}{Theorems}
\crefname{proposition}{Proposition}{Propositions}
\Crefname{proposition}{Proposition}{Propositions}
\crefname{lemma}{Lemma}{Lemmas}
\Crefname{lemma}{Lemma}{Lemmas}
\crefname{corollary}{Corollary}{Corollaries}
\Crefname{corollary}{Corollary}{Corollaries}
\crefname{remark}{Remark}{Remarks}
\Crefname{remark}{Remark}{Remarks}

\DeclareMathOperator*{\esssup}{ess\,sup}
\DeclareMathOperator{\curl}{curl}
\DeclareMathOperator{\diver}{div}

\DeclareMathOperator{\supp}{supp}

\newcommand{\R}{\mathbb R}

\newcommand{\eps}{\varepsilon}

\newcommand{\pv}{\operatorname{p.v.}}

\title{\textbf{Filtered Vortex Stretching and Subgrid Defects\\
		for the Three-Dimensional Navier--Stokes Equations}}

\author{Runlong Yu\\
	The University of Alabama, Tuscaloosa, AL 35487, USA\\
	\texttt{ryu5@ua.edu}}

\date{}

\begin{document}
	
	\maketitle
	
	\begin{abstract}
		We prove a finite-scale estimate for vortex stretching in spatially filtered three-dimensional Navier--Stokes flow.  The positive near-field part of the filtered stretching is bounded by a pairwise defect of filtered vorticity directions.  A magnitude-weighted direction inequality converts this angular defect into a first-order difference quotient of filtered vorticity, and the resulting term is absorbed by filtered diffusion up to a lower-order enstrophy reservoir.  In the localized filtered enstrophy balance, the remaining positive surplus is assigned to far-field strain, commutator forcing, and localization residuals.  The far-field term is reduced to weighted packing and conditional annular Carleson embedding.  The differentiated commutator stress is controlled by a scale-invariant increment defect adapted to the filter and its derivative.  At the critical exponent, bounded increment defects generate cylindrical generalized Young-measure profiles.
	\end{abstract}
	
	
	\medskip
	

\tableofcontents

\section{Introduction}

We consider the three-dimensional incompressible Navier--Stokes equations with unit viscosity,
\begin{equation}\label{eq:NS}
\partial_t u-\Delta u+(u\cdot\nabla)u+\nabla p=0,
\qquad
\diver u=0.
\end{equation}
The finite-energy framework originates in the work of Leray \cite{Leray1934} and Hopf \cite{Hopf1951}, while the suitable-weak and partial-regularity theory was developed by Scheffer \cite{Scheffer1976}, Caffarelli--Kohn--Nirenberg \cite{CKN1982}, and Lin \cite{Lin1998}.  In vorticity form, with $\omega=\curl u$, the equation formally reads
\begin{equation}\label{eq:vorticity-raw}
\partial_t\omega-\Delta\omega+(u\cdot\nabla)\omega
=(\omega\cdot\nabla)u.
\end{equation}
The scalar product of the right-hand side with $\omega$ is the vortex-stretching work.  This term is absent in two dimensions and is the principal nonlinear term capable of amplifying enstrophy in three dimensions.  A basic question is therefore how much of this positive work can be depleted by the geometry of the vorticity field.

The classical geometric route begins from the fact that the strain tensor is a Calder\'on--Zygmund singular integral of vorticity.  Its contraction against the local vorticity direction contains an angular cancellation, a principle behind the vorticity-direction criterion of Constantin and Fefferman \cite{ConstantinFefferman1993}, the Euler constraints of Constantin--Fefferman--Majda \cite{ConstantinFeffermanMajda1996}, later Sobolev and boundary refinements such as \cite{BeiraoBerselli2002}, and localized formulations such as \cite{Grujic2009}.  We also refer to \cite{MajdaBertozzi2002} for the vorticity formulation and Biot--Savart law, to \cite{CalderonZygmund1952} for the singular-integral background, and to \cite{Constantin1994} for a broader geometric perspective.

The difficulty is that this geometric cancellation is singular at precisely the level where one would like to use it.  The vorticity direction is undefined on the zero set of vorticity and may have no pointwise regularity for energy-class solutions.  The near-field strain kernel is singular, the far-field strain is nonlocal, and any coarse-grained formulation necessarily creates commutator residuals.  Thus one has to do more than estimate the stretching term: one should place the estimate inside a localized filtered enstrophy balance and keep track of every positive remainder that is not absorbed by diffusion.

This paper addresses that problem by combining geometric depletion with spatial coarse-graining.  Smooth filtering gives resolved equations and exact commutator identities; see Germano's filtering framework \cite{Germano1992} and smooth coarse-graining formulations such as \cite{EyinkAluie2009}.  For a compactly supported mollifier $\varphi$ and a filter length $\ell>0$, set
\[
U_\ell=\varphi_\ell*u,
\qquad
\Omega_\ell=\curl U_\ell,
\qquad
\mathbb S_\ell=\frac12\bigl(\nabla U_\ell+\nabla U_\ell^{T}\bigr).
\]
The filtered vorticity $\Omega_\ell$ is spatially smooth even when $u$ belongs only to the energy class.  Hence the filtered direction $\xi_\ell=\Omega_\ell/|\Omega_\ell|$, with the standard convention on the zero set, can be used without imposing pointwise regularity on the original vorticity.

The first main result is a finite-scale coercive estimate for the singular near-field part of the filtered stretching.  At physical scale $r$ and relative filter length $\ell$, the estimate has the following scale-normalized form
\begin{equation}\label{eq:intro-nearfield-chain}
\mathcal V_{r,\ell}^{+,\mathrm{near}}
\lesssim
\mathcal A_{r,\ell}^{\mathrm{pair}}
\le
\eta\mathcal P_{r,\ell}^{\rho}
+
C_\eta M_{r,\rho}(u)
\left(\frac r\ell\right)^5
\mathcal O_{r,\ell}.
\end{equation}
Here $\mathcal A_{r,\ell}^{\mathrm{pair}}$ is a pairwise defect of filtered vorticity directions, $\mathcal P_{r,\ell}^{\rho}$ is the localized filtered diffusion, $\mathcal O_{r,\ell}$ is the filtered enstrophy reservoir, and $M_{r,\rho}(u)$ is the local energy bound.  Consequently, for every $\eps\in(0,1)$,
\begin{equation}\label{eq:intro-final}
\mathcal V_{r,\ell}^{+,\mathrm{near}}
\le
(1-\eps)\mathcal P_{r,\ell}^{\rho}
+
C_\eps M_{r,\rho}(u)
\left(\frac r\ell\right)^5
\mathcal O_{r,\ell}.
\end{equation}
At a fixed relative filter scale $\ell=\sigma r$, the coefficient is uniform in the physical scale $r$.  The factor $\sigma^{-5}$ records the energy-class smoothing input and isolates the remaining loss as a filter-ratio issue.

The estimate behind \eqref{eq:intro-nearfield-chain} is purely geometric.  The exact contraction of the strain kernel converts positive near-field stretching into angular incoherence of the filtered vorticity direction.  The magnitude weights already present in the stretching term then convert this angular incoherence into first-order increments of $\Omega_\ell$, avoiding any regularized denominator for $\xi_\ell$ and placing the defect at the diffusive scale controlled by $\nabla\Omega_\ell$.  This part of the argument uses only incompressibility, the Calder\'on--Zygmund representation of strain, and the filtered direction geometry.  The Navier--Stokes dynamics enter when the estimate is inserted into the exact localized filtered enstrophy balance.

After this insertion, the near-field term is closed by diffusion.  Every remaining positive surplus is assigned to one of three explicit residual classes: far-field strain, commutator forcing, or localization.  The far-field term has a different geometry from the singular near field.  A distant vorticity shell acts on a smaller core through a slowly varying external strain.  The direct energy-level estimate gives a weighted $\mathcal W_{3/2}$ packing bound; annular reassignment rewrites the same interaction as a discrete Carleson embedding problem; and the harmonic expansion of the exterior strain isolates the low-order affine jet as the mode that can remain visible across nested scales.  Thus the far-field obstruction is reduced to a concrete harmonic-packing and affine-jet cancellation problem.

The commutator term is controlled at the level selected by the differentiated stress.  With
\[
R_\ell=\varphi_\ell*(u\otimes u)-U_\ell\otimes U_\ell,
\]
the filtered vorticity equation contains the forcing $-\curl\diver R_\ell$.  Thus the subgrid, or equivalently commutator, stress is used here as an exact residual of the filtering operation.  Since differentiating $R_\ell$ differentiates the filter kernel, an increment bound based only on $\varphi_\ell$ does not retain the exact structure of $\nabla\!\cdot R_\ell$.  We therefore introduce a derivative-compatible, scale-invariant envelope which samples velocity increments simultaneously against $\varphi_\ell$ and $\ell|\nabla\varphi_\ell|$.  The resulting commutator estimate replaces a scale-worse energy bound by the increment defect $\widetilde{\mathcal S}_{r,\ell}^{(p)}$.

This derivative-compatible formulation also identifies the correct compactness object for a persistent commutator defect.  Onsager-type coarse-graining and local-defect balances provide scalar fluxes, structure functions, and increment norms \cite{ConstantinETiti1994,DuchonRobert2000,EyinkAluie2009}.  These quantities measure the size of a commutator residual, but they do not by themselves identify the microstructure that carries a nonvanishing defect.  At the critical exponent $p=3$, boundedness of $\widetilde{\mathcal S}^{(3)}$ yields a cylindrical DiPerna--Majda generalized Young-measure profile
\[
(\nu_{x,t},\lambda,\nu_{x,t}^{\infty})
\]
on the derivative-compatible increment space \cite{DiPernaMajda1987,Ball1989,AlibertBouchitte1997}.  This profile separates the resolved increment barycenter, non-Dirac oscillation, and concentration with asymptotic direction.  Since the commutator covariance map has quadratic growth below the controlled quartic density, genuine commutator stress defects force non-Dirac increment oscillation, while strict quartic excess forces oscillation or concentration.  The Young measure is therefore the compactness object on which rigidity, recurrence, and positive defect work can be tested.

\subsection*{Main estimates}
The main estimates are the following.
\begin{enumerate}[label=\textup{(\roman*)},leftmargin=2.2em]
\item \emph{Universal near-field coercivity.}  Positive singular stretching is dominated by a pairwise direction defect and absorbed by filtered diffusion up to a lower-order enstrophy reservoir; see \cref{thm:geometric,thm:coercivity,cor:stretching-coercivity}.
\item \emph{Exact balance reduction.}  The filtered vorticity equation gives an exact filtered enstrophy balance in which the near-field term is closed and every surviving positive surplus is assigned to far-field strain, commutator forcing, or localization residuals; see Proposition~\ref{prop:enstrophy-identity} and \cref{thm:weighted-surplus}.
\item \emph{Far-field harmonic-jet structure.}  Weighted packing is complemented by annular reassignment, conditional unweighted Carleson closure, and a fixed-annulus harmonic route that isolates affine jets as possible recurrent low-order modes; see Proposition~\ref{prop:reassigned-farfield} and \cref{thm:conditional-carleson}.
\item \emph{Derivative-compatible commutator estimate.}  The differentiated commutator stress is controlled by diffusion plus a scale-invariant increment defect adapted jointly to the filter and its derivative; see \cref{thm:increment-comm}.
\item \emph{Banach-valued obstruction profiles.}  Persistent bounded critical increment defects admit cylindrical generalized Young-measure extraction.  Under an additional full-representation hypothesis, unresolved quartic excess or genuine commutator stress defect yields nontrivial increment microstructure; see \cref{thm:young-extraction,prop:comm-stress-defect}.
\end{enumerate}
The result is a finite-scale balance theorem.  It closes the singular near-field stretching term and reduces scale-uniform closure to explicit harmonic, compactness, and localization inputs.  In this sense the paper advances the geometric-depletion and coarse-graining viewpoints from scalar diagnostics toward a set of compactness alternatives for filtered vorticity dynamics.

The paper is organized as follows.  In \cref{sec:setting} we define the filtered quantities and state the near-field coercivity results.  The strain representation and exact direction cancellation are proved in \cref{sec:geometry}.  The algebraic and difference-quotient estimates are developed in \cref{sec:difference}, and the coercivity theorem is proved in \cref{sec:coercivity}.  In \cref{sec:balance} we insert the estimate into the localized filtered enstrophy identity.  Sections~\ref{sec:dyadic}--\ref{sec:weighted-surplus} develop the dyadic chain, far-field packing and annular reassignment, derivative-compatible commutator insertion, and the weighted surplus inequality.  \Cref{sec:obstruction} constructs the commutator obstruction profile and defect-work test.  The final section records the resulting alternatives.

\section{Setting and main results}\label{sec:setting}

\subsection{Cylinders, filters, and local energy}

For $z_0=(x_0,t_0)\in\R^3\times\R$ and $r>0$, set
\[
B_r=B_r(x_0),
\qquad
I_r=(t_0-r^2,t_0),
\qquad
Q_r=B_r\times I_r.
\]
Fix
\begin{equation}\label{eq:rho-range}
0<\rho\le \frac14,
\qquad
0<\ell\le \rho r.
\end{equation}
Let
\begin{equation}\label{eq:mollifier}
\varphi\in C_c^\infty(B_1),
\qquad
\varphi\ge0,
\qquad
\int_{\R^3}\varphi\,dx=1,
\end{equation}
and define
\[
\varphi_\ell(x)=\ell^{-3}\varphi(x/\ell),
\qquad
S_\ell f=\varphi_\ell*f.
\]
For a divergence-free velocity field $u$, set
\begin{equation}\label{eq:filtered-fields}
U_\ell=S_\ell u,
\qquad
\Omega_\ell=\curl U_\ell,
\qquad
\mathbb S_\ell=\frac12\bigl(\nabla U_\ell+\nabla U_\ell^T\bigr).
\end{equation}
We define the filtered vorticity direction by
\begin{equation}\label{eq:direction}
\xi_\ell(x,t)=
\begin{cases}
\dfrac{\Omega_\ell(x,t)}{|\Omega_\ell(x,t)|},&\Omega_\ell(x,t)\ne0,\\[1.2ex]
0,&\Omega_\ell(x,t)=0.
\end{cases}
\end{equation}

For the universal near-field estimates, it is enough to assume
\begin{equation}\label{eq:energy-class-local}
u\in L^\infty(I_r;L^2(\R^3)),
\qquad
\diver u=0,
\end{equation}
and the scale-invariant local energy bound
\begin{equation}\label{eq:M-def}
M_{r,\rho}(u)
:=\frac1r\esssup_{t\in I_r}
\int_{B_{(1+2\rho)r}}|u(x,t)|^2\,dx
<\infty.
\end{equation}
The whole-space assumption in \eqref{eq:energy-class-local} ensures that the Biot--Savart representation of the filtered strain is available without an extension construction. The estimate itself uses only the local quantity \eqref{eq:M-def}, because the mollifier is compactly supported.

Let
\begin{equation}\label{eq:theta}
\vartheta\in C_c^\infty(B_1),
\qquad
0\le\vartheta\le1,
\qquad
\vartheta\equiv1\ \text{on }B_{1/2},
\end{equation}
be radial, and define
\begin{equation}\label{eq:Ctheta}
C_\vartheta
:=\int_{B_1}\frac{\vartheta(w)}{|w|^2}\,dw
\le 4\pi.
\end{equation}
We also fix a filter constant $C_\varphi>0$ such that
\begin{equation}\label{eq:Cphi-def}
\|\curl(\varphi_\ell*f)\|_{L^\infty(\R^3)}
\le C_\varphi\ell^{-5/2}\|f\|_{L^2(\R^3)}
\end{equation}
for every $f\in L^2(\R^3;\R^3)$. Such a constant depends only on finitely many $L^2$ norms of the first derivatives of $\varphi$.

\subsection{Near-field strain and scale-normalized quantities}

The rate-of-strain tensor has a Calder\'on--Zygmund representation in terms of $\Omega_\ell$; the kernel is recorded explicitly in \cref{sec:geometry}. Its near-field part at radius $\rho r$ is
\begin{equation}\label{eq:S-near-def}
\mathbb S_\ell^{\mathrm{near}}(x,t)
:=
\frac{3}{8\pi}\pv\int_{\R^3}
\vartheta\!\left(\frac z{\rho r}\right)
\frac{
(\widehat z\times\Omega_\ell(x-z,t))\otimes\widehat z
+
\widehat z\otimes(\widehat z\times\Omega_\ell(x-z,t))
}{|z|^3}\,dz,
\end{equation}
where $\widehat z=z/|z|$. We set
\begin{equation}\label{eq:S-rem-def}
\mathbb S_\ell^{\mathrm{rem}}
:=\mathbb S_\ell-\mathbb S_\ell^{\mathrm{near}}.
\end{equation}

Let $\chi\in C_c^\infty(Q_r)$ satisfy $0\le\chi\le1$. Define the localized filtered enstrophy and enlarged filtered diffusion by
\begin{equation}\label{eq:O-def}
\mathcal O_{r,\ell}[\chi]
:=
\frac1r\iint_{Q_r}\chi\,|\Omega_\ell|^2\,dx\,dt,
\end{equation}
\begin{equation}\label{eq:P-def}
\mathcal P_{r,\ell}^{\rho}
:=
r\int_{I_r}\int_{B_{(1+\rho)r}}|\nabla\Omega_\ell|^2\,dx\,dt.
\end{equation}
The positive near-field stretching is
\begin{equation}\label{eq:V-near-def}
\mathcal V_{r,\ell}^{+,\mathrm{near}}[\chi]
:=
r\iint_{Q_r}\chi\,
\bigl(\mathbb S_\ell^{\mathrm{near}}\Omega_\ell\cdot\Omega_\ell\bigr)_+
\,dx\,dt.
\end{equation}
Finally, the pairwise direction defect is
\begin{align}
\mathcal A_{r,\ell}^{\mathrm{pair}}[\chi]
:= {}&
r\iint_{Q_r}\chi(x,t)|\Omega_\ell(x,t)|^2
\nonumber\\
&\times
\int_{\R^3}
\vartheta\!\left(\frac z{\rho r}\right)
\frac{
|\xi_\ell(x,t)-\xi_\ell(x-z,t)|\,
|\Omega_\ell(x-z,t)|
}{|z|^3}\,dz\,dx\,dt.
\label{eq:A-def}
\end{align}
The inner integral in \eqref{eq:A-def} may initially be interpreted by monotone truncation at $|z|=\delta$. The coercivity theorem below proves that the resulting limit is finite.

All quantities in \eqref{eq:O-def}--\eqref{eq:A-def} are invariant under the simultaneous Navier--Stokes rescaling of $u$, $r$, and $\ell$.

\subsection{Main estimates}

\begin{theorem}[Filtered near-field geometric depletion]\label{thm:geometric}
Assume \eqref{eq:rho-range}--\eqref{eq:energy-class-local}. Then
\begin{equation}\label{eq:geometric-main}
\mathcal V_{r,\ell}^{+,\mathrm{near}}[\chi]
\le
\frac{3}{8\pi}\,
\mathcal A_{r,\ell}^{\mathrm{pair}}[\chi].
\end{equation}
\end{theorem}

\begin{theorem}[Pairwise defect coercivity]\label{thm:coercivity}
Assume \eqref{eq:rho-range}--\eqref{eq:energy-class-local}. For every $\eta>0$,
\begin{equation}\label{eq:coercivity-main}
\mathcal A_{r,\ell}^{\mathrm{pair}}[\chi]
\le
\eta\,\mathcal P_{r,\ell}^{\rho}
+
\frac{C_\vartheta^2C_\varphi^2\rho^2}{\eta}\,
M_{r,\rho}(u)
\left(\frac r\ell\right)^5
\mathcal O_{r,\ell}[\chi].
\end{equation}
In particular, $\mathcal A_{r,\ell}^{\mathrm{pair}}[\chi]<\infty$ whenever the right-hand side is finite.
\end{theorem}

\begin{corollary}[Near-field stretching-to-diffusion coercivity]\label{cor:stretching-coercivity}
Under the assumptions of \cref{thm:coercivity}, for every $\eps\in(0,1)$,
\begin{equation}\label{eq:stretching-coercivity}
\mathcal V_{r,\ell}^{+,\mathrm{near}}[\chi]
\le
(1-\eps)\,\mathcal P_{r,\ell}^{\rho}
+
C_{\eps,\vartheta,\varphi}\,
\rho^2 M_{r,\rho}(u)
\left(\frac r\ell\right)^5
\mathcal O_{r,\ell}[\chi],
\end{equation}
where one may take
\begin{equation}\label{eq:Ceps-explicit}
C_{\eps,\vartheta,\varphi}
=
\frac{9C_\vartheta^2C_\varphi^2}{64\pi^2(1-\eps)}.
\end{equation}
\end{corollary}

\begin{proof}
Combine \cref{thm:geometric,thm:coercivity} and choose
\[
\eta=\frac{8\pi}{3}(1-\eps).
\]
\end{proof}

\begin{corollary}[Fixed relative filter scale]\label{cor:fixed-ratio}
Let $\ell=\sigma r$ with $0<\sigma\le\rho$. If $M_{r,\rho}(u)\le M$, then
\begin{equation}\label{eq:fixed-ratio}
\mathcal V_{r,\sigma r}^{+,\mathrm{near}}[\chi]
\le
(1-\eps)\mathcal P_{r,\sigma r}^{\rho}
+
C_{\eps,\vartheta,\varphi}\rho^2M\sigma^{-5}
\mathcal O_{r,\sigma r}[\chi].
\end{equation}
The constant is uniform in the physical scale $r$.
\end{corollary}

\begin{remark}[Nature of the loss]\label{rem:sigma-loss}
The estimate is invariant under simultaneous parabolic scaling of $r$ and $\ell$. The loss is instead a degeneration in the relative filter parameter $\sigma=\ell/r$: the coefficient behaves like $\sigma^{-5}$ as $\sigma\downarrow0$. This is the square of the basic energy-class smoothing estimate $\|\Omega_\ell\|_\infty\lesssim\ell^{-5/2}\|u\|_2$.
\end{remark}

\begin{remark}[Kinematic universality]\label{rem:kinematic}
The near-field depletion and coercivity estimates hold for every divergence-free energy-class field on $\R^3$.  Their proof uses the singular-integral geometry of strain and the magnitude-weighted direction inequality, independently of the evolution equation.  The Navier--Stokes dynamics enter when this universal geometric estimate is placed in the localized filtered enstrophy balance; see \cref{sec:balance}.
\end{remark}

\section{Strain representation and exact geometric depletion}\label{sec:geometry}

\subsection{The strain kernel}

We record the classical singular-integral representation used in geometric vorticity arguments.

\begin{lemma}[Calder\'on--Zygmund representation of strain]\label{lem:strain-kernel}
Let $U\in L^2(\R^3;\R^3)$ be divergence free and smooth, let $\Omega=\curl U$, and set
\[
\mathbb S=\frac12(\nabla U+\nabla U^T).
\]
Then
\begin{equation}\label{eq:strain-index}
\mathbb S_{ij}(x)
=
\pv\int_{\R^3}K_{ijm}(z)\,\Omega_m(x-z)\,dz,
\end{equation}
where
\begin{equation}\label{eq:K-index}
K_{ijm}(z)
=
\frac{3}{8\pi|z|^5}
\bigl(
 z_j\varepsilon_{ikm}z_k
+
 z_i\varepsilon_{jkm}z_k
\bigr).
\end{equation}
Equivalently,
\begin{equation}\label{eq:strain-vector}
\mathbb S(x)
=
\frac{3}{8\pi}\pv\int_{\R^3}
\frac{
(\widehat z\times\Omega(x-z))\otimes\widehat z
+
\widehat z\otimes(\widehat z\times\Omega(x-z))
}{|z|^3}\,dz.
\end{equation}
The kernel is homogeneous of degree $-3$, smooth away from the origin, and has zero spherical average.
\end{lemma}

\begin{proof}
For smooth divergence-free $U$, the Biot--Savart law gives
\[
U(x)=\frac1{4\pi}\int_{\R^3}
\Omega(y)\times\frac{x-y}{|x-y|^3}\,dy.
\]
Differentiating in $x$, symmetrizing, and observing that the antisymmetric Kronecker term cancels gives \eqref{eq:strain-index}. Formula \eqref{eq:strain-vector} is the invariant form of the same identity. The Calder\'on--Zygmund properties follow immediately from \eqref{eq:K-index}; see the classical theory \cite{CalderonZygmund1952}. A standard approximation argument extends the formula to smooth $L^2$ fields. See also \cite[Chapter 2]{MajdaBertozzi2002}.
\end{proof}

Because spatial filtering maps $L^2$ into $C_b^\infty$, \cref{lem:strain-kernel} applies to $U_\ell(\cdot,t)$ for almost every $t$.

\subsection{Direction contraction and cancellation}

\begin{lemma}[Exact direction contraction]\label{lem:direction-contraction}
Fix $(x,t)$ with $\Omega_\ell(x,t)\ne0$ and put $\xi=\xi_\ell(x,t)$. For every truncation parameter $\delta>0$,
\begin{align}
\xi\cdot\mathbb S_{\ell,\delta}^{\mathrm{near}}(x,t)\xi
={}&
\frac{3}{4\pi}
\int_{|z|>\delta}
\vartheta\!\left(\frac z{\rho r}\right)
\frac{(\xi\cdot\widehat z)
(\xi\times\widehat z)\cdot
\bigl(\xi_\ell(x-z,t)-\xi\bigr)}{|z|^3}
\nonumber\\
&\hspace{6em}\times
|\Omega_\ell(x-z,t)|\,dz,
\label{eq:exact-direction-contraction}
\end{align}
where $\mathbb S_{\ell,\delta}^{\mathrm{near}}$ denotes \eqref{eq:S-near-def} with the region $|z|\le\delta$ removed.
Moreover,
\begin{equation}\label{eq:angular-factor}
\frac{3}{4\pi}|\xi\cdot\widehat z|\,|\xi\times\widehat z|
\le \frac{3}{8\pi}.
\end{equation}
\end{lemma}

\begin{proof}
Contracting \eqref{eq:strain-vector} with $\xi$ on both sides yields
\[
\xi\cdot\mathbb S_{\ell,\delta}^{\mathrm{near}}\xi
=
\frac{3}{4\pi}
\int_{|z|>\delta}
\vartheta\!\left(\frac z{\rho r}\right)
\frac{(\xi\cdot\widehat z)
\xi\cdot(\widehat z\times\Omega_\ell(x-z,t))}{|z|^3}\,dz.
\]
Write $\Omega_\ell(x-z,t)=|\Omega_\ell(x-z,t)|\xi_\ell(x-z,t)$. Since
\[
(\xi\times\widehat z)\cdot\xi=0,
\]
one may subtract $\xi$ exactly inside the last factor, proving \eqref{eq:exact-direction-contraction}. Inequality \eqref{eq:angular-factor} follows from
\[
|\xi\cdot\widehat z|\,|\xi\times\widehat z|
=|\cos\theta\sin\theta|\le\frac12.
\]
\end{proof}

\subsection{Proof of the geometric depletion theorem}

\begin{proof}[Proof of \cref{thm:geometric}]
At points where $\Omega_\ell(x,t)=0$, both sides of the pointwise stretching estimate below vanish. At points where $\Omega_\ell(x,t)\ne0$, \cref{lem:direction-contraction} gives, for every $\delta>0$,
\begin{align}
\bigl(
\mathbb S_{\ell,\delta}^{\mathrm{near}}\Omega_\ell\cdot\Omega_\ell
\bigr)_+
\le{}&
\frac{3}{8\pi}|\Omega_\ell(x,t)|^2
\int_{|z|>\delta}
\vartheta\!\left(\frac z{\rho r}\right)
\nonumber\\
&\times
\frac{
|\xi_\ell(x,t)-\xi_\ell(x-z,t)|
|\Omega_\ell(x-z,t)|
}{|z|^3}\,dz.
\label{eq:pointwise-geometric-truncated}
\end{align}
Multiply by $r\chi$ and integrate. The right-hand side is the truncated version of $\frac{3}{8\pi}\mathcal A_{r,\ell}^{\mathrm{pair}}[\chi]$. As $\delta\downarrow0$, the principal-value convergence theorem for Calder\'on--Zygmund operators gives almost-everywhere convergence on the left. Fatou's lemma for the nonnegative stretching integrand and monotone convergence for the truncated defect then give \eqref{eq:geometric-main}. If the limiting defect is infinite, the conclusion is understood in the extended sense and is immediate.
\end{proof}

\section{Direction increments and local difference quotients}\label{sec:difference}

\subsection{An algebraic direction inequality}

The zero set of $\Omega_\ell$ creates no difficulty because the magnitude weights in \eqref{eq:A-def} compensate for the singularity of the direction map.

\begin{lemma}[Magnitude-weighted direction difference]\label{lem:algebraic-direction}
For $a,b\in\R^3$, define
\[
\widehat a=
\begin{cases}a/|a|,&a\ne0,\\0,&a=0,\end{cases}
\qquad
\widehat b=
\begin{cases}b/|b|,&b\ne0,\\0,&b=0.\end{cases}
\]
Then
\begin{equation}\label{eq:min-direction}
\min\{|a|,|b|\}\,|\widehat a-\widehat b|
\le 2|a-b|.
\end{equation}
Consequently, if $|a|,|b|\le\Lambda$, then
\begin{equation}\label{eq:weighted-direction}
|a|^2|b|\,|\widehat a-\widehat b|
\le 2\Lambda|a|\,|a-b|.
\end{equation}
\end{lemma}

\begin{proof}
If either vector vanishes, \eqref{eq:min-direction} is immediate. Suppose that $0<|a|\le|b|$. Then
\begin{align*}
|a|\,|\widehat a-\widehat b|
&=|a-|a|\widehat b|\\
&\le |a-b|+|b-|a|\widehat b|\\
&=|a-b|+\bigl||b|-|a|\bigr|\\
&\le2|a-b|.
\end{align*}
The case $|b|\le|a|$ is symmetric. Multiplying \eqref{eq:min-direction} by $|a|^2|b|/\min\{|a|,|b|\}$ gives
\[
|a|^2|b|\,|\widehat a-\widehat b|
\le2|a|\max\{|a|,|b|\}|a-b|,
\]
which implies \eqref{eq:weighted-direction}.
\end{proof}

\subsection{A local difference-quotient operator}

For a vector field $f$ define
\begin{equation}\label{eq:N-def}
N_{\rho r}^{\vartheta}f(x)
:=
\int_{\R^3}
\vartheta\!\left(\frac z{\rho r}\right)
\frac{|f(x)-f(x-z)|}{|z|^3}\,dz.
\end{equation}

\begin{lemma}[Local $L^2$ difference-quotient estimate]\label{lem:difference-quotient}
If $f\in H^1(B_{(1+\rho)r};\R^m)$, then
\begin{equation}\label{eq:N-L2}
\|N_{\rho r}^{\vartheta}f\|_{L^2(B_r)}
\le
\rho r C_\vartheta
\|\nabla f\|_{L^2(B_{(1+\rho)r})}.
\end{equation}
\end{lemma}

\begin{proof}
By Minkowski's integral inequality,
\begin{align*}
\|N_{\rho r}^{\vartheta}f\|_{L^2(B_r)}
&\le
\int_{\R^3}
\vartheta\!\left(\frac z{\rho r}\right)|z|^{-3}
\|f(\cdot)-f(\cdot-z)\|_{L^2(B_r)}\,dz.
\end{align*}
For $|z|\le\rho r$, the fundamental theorem of calculus gives
\[
f(x)-f(x-z)
=\int_0^1 z\cdot\nabla f(x-sz)\,ds,
\]
and therefore
\[
\|f(\cdot)-f(\cdot-z)\|_{L^2(B_r)}
\le |z|\|\nabla f\|_{L^2(B_{(1+\rho)r})}.
\]
Substitution and the change of variables $z=\rho rw$ yield
\[
\|N_{\rho r}^{\vartheta}f\|_{L^2(B_r)}
\le
\left(
\int_{\R^3}\vartheta\!\left(\frac z{\rho r}\right)|z|^{-2}\,dz
\right)
\|\nabla f\|_{L^2(B_{(1+\rho)r})}
=
\rho rC_\vartheta\|\nabla f\|_2.
\]
\end{proof}

\begin{proposition}[Defect bound in terms of $\|\Omega_\ell\|_\infty$]\label{prop:A-Lambda}
Let
\begin{equation}\label{eq:Lambda-def}
\Lambda_{r,\ell}^{\rho}
:=
\|\Omega_\ell\|_{L^\infty(B_{(1+\rho)r}\times I_r)}.
\end{equation}
Then
\begin{equation}\label{eq:A-Lambda}
\mathcal A_{r,\ell}^{\mathrm{pair}}[\chi]
\le
2\rho C_\vartheta r^2\Lambda_{r,\ell}^{\rho}
\bigl(
\mathcal O_{r,\ell}[\chi]
\mathcal P_{r,\ell}^{\rho}
\bigr)^{1/2}.
\end{equation}
Consequently, for every $\eta>0$,
\begin{equation}\label{eq:A-Lambda-Young}
\mathcal A_{r,\ell}^{\mathrm{pair}}[\chi]
\le
\eta\mathcal P_{r,\ell}^{\rho}
+
\frac{\rho^2C_\vartheta^2}{\eta}
r^4(\Lambda_{r,\ell}^{\rho})^2
\mathcal O_{r,\ell}[\chi].
\end{equation}
\end{proposition}

\begin{proof}
Apply \eqref{eq:weighted-direction} with
\[
a=\Omega_\ell(x,t),
\qquad
b=\Omega_\ell(x-z,t),
\qquad
\Lambda=\Lambda_{r,\ell}^{\rho}.
\]
Because $x\in B_r$ and $\vartheta(z/(\rho r))\ne0$ imply $x-z\in B_{(1+\rho)r}$, we obtain
\begin{align}
\mathcal A_{r,\ell}^{\mathrm{pair}}[\chi]
\le{}&
2r\Lambda_{r,\ell}^{\rho}
\iint_{Q_r}
\chi(x,t)|\Omega_\ell(x,t)|
N_{\rho r}^{\vartheta}\Omega_\ell(x,t)
\,dx\,dt.
\label{eq:A-before-Cauchy}
\end{align}
Cauchy--Schwarz in space-time and \cref{lem:difference-quotient} give
\begin{align*}
\mathcal A_{r,\ell}^{\mathrm{pair}}[\chi]
&\le
2\rho C_\vartheta r^2\Lambda_{r,\ell}^{\rho}
\left(\iint_{Q_r}\chi|\Omega_\ell|^2\right)^{1/2}
\left(\int_{I_r}\int_{B_{(1+\rho)r}}|\nabla\Omega_\ell|^2\right)^{1/2}\\
&=
2\rho C_\vartheta r^2\Lambda_{r,\ell}^{\rho}
\bigl(\mathcal O_{r,\ell}[\chi]\mathcal P_{r,\ell}^{\rho}\bigr)^{1/2}.
\end{align*}
This proves \eqref{eq:A-Lambda}. Young's inequality in the form $2ab\le\eta b^2+\eta^{-1}a^2$ yields \eqref{eq:A-Lambda-Young}.
\end{proof}

\section{Filter smoothing and proof of coercivity}\label{sec:coercivity}

\begin{lemma}[Local filtered vorticity bound]\label{lem:filter-smoothing}
Under \eqref{eq:rho-range} and \eqref{eq:M-def},
\begin{equation}\label{eq:Lambda-smoothing}
r^2\Lambda_{r,\ell}^{\rho}
\le
C_\varphi M_{r,\rho}(u)^{1/2}
\left(\frac r\ell\right)^{5/2}.
\end{equation}
\end{lemma}

\begin{proof}
For $x\in B_{(1+\rho)r}$ and almost every $t\in I_r$,
\[
\Omega_\ell(x,t)=\curl(\varphi_\ell*u)(x,t).
\]
Since $\supp\varphi_\ell\subset B_\ell$ and $\ell\le\rho r$, Young's inequality on the local convolution region gives
\begin{align*}
|\Omega_\ell(x,t)|
&\le
C_\varphi\ell^{-5/2}
\|u(\cdot,t)\|_{L^2(B_{(1+2\rho)r})}\\
&\le
C_\varphi\ell^{-5/2}
M_{r,\rho}(u)^{1/2}r^{1/2}.
\end{align*}
Multiplying by $r^2$ gives \eqref{eq:Lambda-smoothing}.
\end{proof}

\begin{proof}[Proof of \cref{thm:coercivity}]
Insert \eqref{eq:Lambda-smoothing} into \eqref{eq:A-Lambda-Young}. Since
\[
r^4(\Lambda_{r,\ell}^{\rho})^2
\le
C_\varphi^2M_{r,\rho}(u)
\left(\frac r\ell\right)^5,
\]
we obtain \eqref{eq:coercivity-main}.
\end{proof}

\subsection{Scaling check}

For completeness, under the parabolic rescaling centered at $z_0$,
\[
u^{(r)}(y,s)=r u(x_0+ry,t_0+r^2s),
\qquad
p^{(r)}(y,s)=r^2p(x_0+ry,t_0+r^2s),
\]
the filtered vorticity satisfies
\[
\Omega_{\ell}^{(r)}(y,s)
=r^2\Omega_\ell(x_0+ry,t_0+r^2s),
\]
with relative filter length $\sigma=\ell/r$. Direct substitution shows that $\mathcal A$, $\mathcal P$, $\mathcal O$, and $M_{r,\rho}$ are dimensionless. Hence the only parameter generated by the filter is the dimensionless ratio $r/\ell$.

\section{Localized filtered enstrophy balance}\label{sec:balance}

We now place the universal near-field coercivity estimate inside the Navier--Stokes dynamics.  The exact filtered vorticity identity supplies a balance in which geometric depletion, diffusion, commutator forcing, and localization appear as separate terms.

Let $(u,p)$ be a whole-space suitable weak solution of \eqref{eq:NS} on a time interval containing $I_r$. Define
\begin{equation}\label{eq:filtered-pressure-stress}
P_\ell=\varphi_\ell*p,
\qquad
R_\ell=\varphi_\ell*(u\otimes u)-U_\ell\otimes U_\ell.
\end{equation}
Spatial filtering gives
\begin{equation}\label{eq:filtered-momentum}
\partial_tU_\ell-\Delta U_\ell
+\diver(U_\ell\otimes U_\ell)+\nabla P_\ell
=-\diver R_\ell,
\qquad
\diver U_\ell=0.
\end{equation}
Taking curl yields
\begin{equation}\label{eq:filtered-vorticity}
\partial_t\Omega_\ell-\Delta\Omega_\ell
+(U_\ell\cdot\nabla)\Omega_\ell
=(\Omega_\ell\cdot\nabla)U_\ell+F_\ell,
\qquad
F_\ell:=-\curl\diver R_\ell.
\end{equation}
Because the antisymmetric part of $\nabla U_\ell$ does not contribute,
\begin{equation}\label{eq:stretching-strain}
(\Omega_\ell\cdot\nabla)U_\ell\cdot\Omega_\ell
=\mathbb S_\ell\Omega_\ell\cdot\Omega_\ell.
\end{equation}

\begin{proposition}[Localized filtered enstrophy identity]\label{prop:enstrophy-identity}
Let $s_0<s_1$ with $[s_0,s_1]\subset I_r$, and let
$\chi\in C^\infty(\R^3\times[s_0,s_1])$ be nonnegative and compactly supported in the spatial variable. Then
\begin{equation}\label{eq:enstrophy-identity}
\mathcal E_\chi(s_1)-\mathcal E_\chi(s_0)
+\mathcal P_\chi
=
\mathcal V_\chi^{\mathrm{near}}
+\mathcal V_\chi^{\mathrm{rem}}
+\mathcal R_\chi
+\mathcal L_\chi,
\end{equation}
where
\begin{align}
\mathcal E_\chi(t)
&:=\frac r2\int_{\R^3}\chi(x,t)|\Omega_\ell(x,t)|^2\,dx,
\label{eq:Echi}\\
\mathcal P_\chi
&:=r\int_{s_0}^{s_1}\int_{\R^3}\chi|\nabla\Omega_\ell|^2\,dx\,dt,
\label{eq:Pchi}\\
\mathcal V_\chi^{\mathrm{near}}
&:=r\int_{s_0}^{s_1}\int_{\R^3}
\chi\,\mathbb S_\ell^{\mathrm{near}}\Omega_\ell\cdot\Omega_\ell\,dx\,dt,
\label{eq:Vchi-near}\\
\mathcal V_\chi^{\mathrm{rem}}
&:=r\int_{s_0}^{s_1}\int_{\R^3}
\chi\,\mathbb S_\ell^{\mathrm{rem}}\Omega_\ell\cdot\Omega_\ell\,dx\,dt,
\label{eq:Vchi-rem}\\
\mathcal R_\chi
&:=r\int_{s_0}^{s_1}\int_{\R^3}\chi F_\ell\cdot\Omega_\ell\,dx\,dt,
\label{eq:Rchi}\\
\mathcal L_\chi
&:=\frac r2\int_{s_0}^{s_1}\int_{\R^3}
|\Omega_\ell|^2
(\partial_t\chi+\Delta\chi+U_\ell\cdot\nabla\chi)
\,dx\,dt.
\label{eq:Lchi}
\end{align}
\end{proposition}

\begin{proof}
Multiply \eqref{eq:filtered-vorticity} by $r\chi\Omega_\ell$ and integrate over $\R^3\times(s_0,s_1)$. Integration by parts in time and space gives the endpoint term and \eqref{eq:Lchi}; incompressibility gives the transport contribution $\frac r2\int|\Omega_\ell|^2U_\ell\cdot\nabla\chi$. Formula \eqref{eq:stretching-strain} and the decomposition \eqref{eq:S-rem-def} give the two stretching terms. For a suitable weak solution, the identity follows rigorously by an additional time mollification and passage to the limit.
\end{proof}

To display the exact localization cost created by the enlarged diffusion region, assume now that $0\le\chi\le1$ and $\supp\chi\subset B_r\times(s_0,s_1)$. Define
\begin{equation}\label{eq:annular-diffusion}
\mathcal L_{\mathrm{ann}}
:=
\mathcal P_{r,\ell}^{\rho}-\mathcal P_\chi
\ge0,
\end{equation}
where $\mathcal P_{r,\ell}^{\rho}$ is restricted to the time interval $(s_0,s_1)$ in this section. Also define
\begin{equation}\label{eq:Vrem-positive}
\mathcal V_\chi^{+,\mathrm{rem}}
:=r\int_{s_0}^{s_1}\int_{\R^3}
\chi\bigl(\mathbb S_\ell^{\mathrm{rem}}\Omega_\ell\cdot\Omega_\ell\bigr)_+
\,dx\,dt.
\end{equation}

\begin{corollary}[Coercive near-field insertion]\label{cor:balance-insertion}
Under the assumptions above, for every $\eps\in(0,1)$,
\begin{align}
\mathcal E_\chi(s_1)+\eps\mathcal P_\chi
\le{}&
\mathcal E_\chi(s_0)
+
C_{\eps,\vartheta,\varphi}\rho^2M_{r,\rho}(u)
\left(\frac r\ell\right)^5
\mathcal O_{r,\ell}[\chi]
\nonumber\\
&+
(1-\eps)\mathcal L_{\mathrm{ann}}
+
\mathcal V_\chi^{+,\mathrm{rem}}
+
|\mathcal R_\chi|
+
|\mathcal L_\chi|.
\label{eq:balance-coercive}
\end{align}
\end{corollary}

\begin{proof}
From \cref{cor:stretching-coercivity},
\[
\mathcal V_\chi^{\mathrm{near}}
\le
\mathcal V_{r,\ell}^{+,\mathrm{near}}[\chi]
\le
(1-\eps)\mathcal P_{r,\ell}^{\rho}
+
C_{\eps,\vartheta,\varphi}\rho^2M_{r,\rho}(u)
\left(\frac r\ell\right)^5
\mathcal O_{r,\ell}[\chi].
\]
Use \eqref{eq:annular-diffusion} to write
\[
\mathcal P_{r,\ell}^{\rho}
=\mathcal P_\chi+\mathcal L_{\mathrm{ann}},
\]
and insert the result into \eqref{eq:enstrophy-identity}. Bound the remaining signed terms by their positive or absolute parts.
\end{proof}

\begin{remark}[Closure of the singular near-field term]
Estimate \eqref{eq:balance-coercive} absorbs positive singular near-field stretching into diffusion, up to a lower-order filtered-enstrophy reservoir and the explicit annular diffusion budget.  All other signed work is assigned to the named quantities $\mathcal V_\chi^{+,\mathrm{rem}}$, $\mathcal R_\chi$, and $\mathcal L_\chi$, representing far-field strain, commutator forcing, and cutoff/transport residual.  Sections~\ref{sec:farfield} and \ref{sec:comm} sharpen the first two terms, while Proposition~\ref{prop:adjoint-residual} identifies an exact cancellation of the principal cutoff residual.
\end{remark}

\begin{proposition}[Adjoint cancellation of the localization residual]\label{prop:adjoint-residual}
Assume that the filtered velocity $U_\ell$ is fixed and that $\chi$ is a smooth nonnegative solution of the backward adjoint drift-diffusion equation
\begin{equation}\label{eq:adjoint-cutoff}
\partial_t\chi+\Delta\chi+U_\ell\cdot\nabla\chi=0
\end{equation}
on $\mathbb R^3\times(s_0,s_1)$, or on $\mathbb T^3\times(s_0,s_1)$.  Then the localization residual in Proposition~\ref{prop:enstrophy-identity} vanishes:
\begin{equation}\label{eq:adjoint-residual-zero}
\mathcal L_\chi=0.
\end{equation}
Equivalently, one may prescribe terminal data at $s_1$ and solve the backward equation $-\partial_t\chi-\Delta\chi-U_\ell\cdot\nabla\chi=0$.  If $0\le \chi(\cdot,s_1)\le1$, then $0\le\chi\le1$ by the maximum principle.
\end{proposition}

\begin{proof}
By definition,
\[
\mathcal L_\chi=\frac r2\int_{s_0}^{s_1}\!\int |\Omega_\ell|^2
(\partial_t\chi+\Delta\chi+U_\ell\cdot\nabla\chi)\,dx\,dt.
\]
Equation \eqref{eq:adjoint-cutoff} makes the integrand vanish.  The maximum principle applies because $U_\ell$ is smooth in space and divergence free after filtering.
\end{proof}

\begin{remark}[Localization architecture]\label{rem:localization-residual-scope}
Proposition~\ref{prop:adjoint-residual} closes the principal cutoff residual in the exact whole-space filtering setting through a solution-adapted adjoint weight.  The shell terms generated by integration by parts in the commutator term remain as explicit nonnegative localization budgets.  A purely local-cylinder realization naturally adds a solenoidal extension or localized filtering procedure, together with its extension and filter-transition budgets; these terms form a separate localization module.
\end{remark}

\section{Dyadic weighted formulation}\label{sec:dyadic}

We now pass from a one-scale coercivity estimate to a finite-chain statement.  This section is formulated in the exact whole-space setting, where no extension residual is present.  The local interior case requires the extension and filter-transition terms described in \cref{sec:discussion}.

Fix a base cylinder with radius normalized to $r_0=1$ and let
\[
r_k=2^{-k},\qquad \ell_k=\sigma r_k,
\qquad 0<\sigma<1,
\qquad k=0,1,\dots,N.
\]
Let $Q_k=Q_{r_k}(z_0)$ and choose smooth cutoffs $\chi_k$ obtained by parabolic rescaling of one fixed profile.  Write
\[
U_k=U_{\ell_k},\qquad
\Omega_k=\Omega_{\ell_k},\qquad
R_k=R_{\ell_k}=S_{\ell_k}(u\otimes u)-U_k\otimes U_k.
\]
We use the normalized quantities
\[
P_k:=r_k\iint_{Q_k}\chi_k|\nabla\Omega_k|^2,
\qquad
\mathcal O_k:=r_k^{-1}\iint_{Q_k}\chi_k|\Omega_k|^2,
\]
\[
V_k^+:=r_k\iint_{Q_k}\chi_k(\mathbb S_{\ell_k}\Omega_k\cdot\Omega_k)_+.
\]
Let
\[
V_k^{+,\mathrm{near}},\qquad V_k^{+,\mathrm{far}}
\]
denote the contributions after splitting the strain kernel at distance $\rho r_k$.  The commutator forcing is
\begin{equation}\label{eq:Fcomm-def}
F_k^{\mathrm{com}}
:=
r_k\left|
\iint_{Q_k}\chi_k\,
\Omega_k\cdot\bigl(\nabla\times\nabla\!\cdot R_k\bigr)
\right|.
\end{equation}
The localization residual term $L_k$ denotes the sum of the cutoff residual from the filtered enstrophy identity, the annular diffusion cost generated by enlarged cutoffs, and the shell terms from the commutator insertion.  By definition, $L_k$ contains only nonnegative localization budgets and no coercive contribution.

For $\alpha>0$ define the weighted sequence class
\begin{equation}\label{eq:Walpha-def}
\mathcal W_\alpha
:=
\left\{w=(w_k)_{k\ge0}: w_k\ge0,
\quad
\|w\|_{\mathcal W_\alpha}:=
\sum_{k=0}^{\infty}2^{\alpha k}w_k<\infty
\right\}.
\end{equation}
Thus $w_k=2^{-\beta k}$ belongs to $\mathcal W_\alpha$ precisely when $\beta>\alpha$.

Let
\begin{equation}\label{eq:ME-def}
M_E:=\esssup_{t}\|u(t)\|_{L^2(\R^3)}^2
\end{equation}
in the whole-space case.  The estimates below are written with $M_E$ to emphasize that this part of the paper uses only energy-level information.

\section{Far-field strain packing and annular reassignment}\label{sec:farfield}

The near-field term is controlled by \cref{cor:stretching-coercivity}.  The far-field term has a complementary geometry: vorticity outside the core generates a slowly varying external strain on the smaller cylinder.  A direct energy-level treatment gives a robust weighted packing estimate from the Calder\'on--Zygmund tail and the $L^2$-to-$L^\infty$ filter bound.  Annular reassignment then resolves this tail into interactions between coarse vorticity reservoirs and core enstrophy profiles, producing a discrete Carleson embedding structure.  Finally, the harmonic-jet expansion identifies the low-order affine modes which should cancel, decorrelate, or be summably packed in an unweighted theory.

\begin{lemma}[Far-field tail kernel]\label{lem:tail-kernel}
Let $K$ be any component of the three-index strain kernel in \cref{lem:strain-kernel}.  Then
\begin{equation}\label{eq:tail-kernel}
\|K\mathbf 1_{|z|>\rho r}\|_{L^2(\R^3)}
\le C_K\rho^{-3/2}r^{-3/2},
\end{equation}
where $C_K$ is universal.
\end{lemma}

\begin{proof}
The kernel satisfies $|K(z)|\le C|z|^{-3}$.  Therefore
\[
\int_{|z|>\rho r}|K(z)|^2\,dz
\le
C\int_{\rho r}^{\infty}s^{-6}s^2\,ds
=
C(\rho r)^{-3}.
\]
Taking the square root proves the claim.
\end{proof}

\begin{theorem}[Far-field weighted packing]\label{thm:farfield-packing}
Assume $u$ is a Leray--Hopf solution on $\R^3$, normalize $r_0=1$, set $r_k=2^{-k}$ and $\ell_k=\sigma r_k$, and let $V_k^{+,\mathrm{far}}$ be the positive far-field stretching contribution.  Then
\begin{equation}\label{eq:Ffar-single}
V_k^{+,\mathrm{far}}
\le
C_{\rho,\sigma,\chi,\varphi}M_E^{3/2}2^{3k/2}
\end{equation}
for all $k\ge0$.  Consequently, for every $w\in\mathcal W_{3/2}$,
\begin{equation}\label{eq:Ffar-pack}
\sum_{k=0}^{N}w_kV_k^{+,\mathrm{far}}
\le
C_{\rho,\sigma,\chi,\varphi}M_E^{3/2}
\|w\|_{\mathcal W_{3/2}}.
\end{equation}
\end{theorem}

\begin{proof}
By \cref{lem:tail-kernel} and Young's inequality for convolution,
\[
\|\mathbb S_{\ell_k}^{\mathrm{far}}(t)\|_{L^\infty(B_{r_k})}
\le
C\rho^{-3/2}r_k^{-3/2}
\|\Omega_k(t)\|_{L^2(B_{2r_k})}.
\]
The filter smoothing estimate gives
\[
\|\Omega_k(t)\|_{L^\infty}
\le C_\varphi \ell_k^{-5/2}\|u(t)\|_{L^2}
\le C_\varphi \sigma^{-5/2}r_k^{-5/2}M_E^{1/2}.
\]
Estimating the local $L^2$ mass of $\Omega_k$ by its $L^\infty$ norm over a ball of volume $O(r_k^3)$ and integrating over a time interval of length $O(r_k^2)$ gives
\[
V_k^{+,\mathrm{far}}
\le
C_{\rho,\sigma,\chi,\varphi}M_E^{3/2}r_k^{-3/2}.
\]
Since $r_k=2^{-k}$, this is \eqref{eq:Ffar-single}.  Summing against $w_k$ gives \eqref{eq:Ffar-pack}.
\end{proof}

\begin{remark}[Energy-level baseline and geometric refinement]\label{rem:farfield-weight-loss}
The factor $2^{3k/2}$ is the energy-level baseline for the nonlocal tail.  The annular and harmonic formulations below locate the source of this loss more precisely: high-order external modes gain scale separation, while low-order strain jets can recur across nested cores.  This exposes two concrete routes to unweighted closure---Carleson summability of reassigned reservoirs or cancellation of the affine jet.
\end{remark}

\subsection{Annular shell formula and reassignment}

Let
\[
I_k=(t_0-r_k^2,t_0).
\]
Let $\mu_k^{\mathrm{far,ann}}$ denote the contribution of the shells $0\le m\le k$ to $V_k^{+,\mathrm{far}}$.  The remaining exterior shells $m>k$ are not included in the reassigned annular quantity; they are kept in the energy-level far-field budget unless an additional exterior-tail estimate is imposed.
For $x\in B_{r_k}(x_0)$ split the far-field strain into shells
\[
\mathbb S_k^{\mathrm{far}}(x,t)=\sum_{m\ge0}\mathbb S_{k,m}(x,t),
\qquad
\mathbb S_{k,m}(x,t)
:=
\int_{A_{k,m}(x)}K(x-y)\Omega_k(y,t)\,dy,
\]
where
\[
A_{k,m}(x):=\{y:\Gamma 2^m r_k<|x-y|\le 2\Gamma 2^m r_k\},
\qquad \Gamma>1,
\]
and let $A_{k,m}^*$ denote the union of $A_{k,m}(x)$ over $x\in B_{r_k}(x_0)$.  The off-diagonal bound $|K(z)|\lesssim |z|^{-3}$ and Cauchy--Schwarz on the shell give
\begin{equation}\label{eq:far-shell-estimate}
\mu_k^{\mathrm{far,ann}}
\le
C_0(\Gamma,\sigma,\chi,\varphi)
\sum_{m=0}^{k}2^{-3m/2}
\mathcal E_{k,m}^{1/2}\,\mathcal Q_k,
\end{equation}
where
\begin{equation}\label{eq:Ekm-Qk-def}
\mathcal E_{k,m}:=
 r_k^{-1}\iint_{I_k\times A_{k,m}^*}|\Omega_k|^2,
\qquad
\mathcal Q_k:=
\left(
 \int_{I_k}
 \left(\int_{B_{2r_k}}\chi_k|\Omega_k|^2\,dx\right)^2dt
\right)^{1/2}.
\end{equation}
The quantity $\mathcal Q_k$ is a time-profile version of the local filtered enstrophy reservoir.  Its normalization is chosen so that the shell formula has the same scale as $\mu_k^{\mathrm{far,ann}}$.

\begin{lemma}[Annular reassignment with bounded overlap]\label{lem:annular-reassignment}
Let $j=k-m$ with $0\le m\le k$, and define the stationary annulus
\begin{equation}\label{eq:Ajtilde-def}
\widetilde A_j
:=
\{y:(\Gamma-1)r_j<|y-x_0|\le (2\Gamma+1)r_j\}.
\end{equation}
Then
\[
A_{k,m}(x)\subset \widetilde A_{k-m}
\qquad\text{for every }x\in B_{r_k}(x_0),
\]
and hence $A_{k,m}^*\subset \widetilde A_{k-m}$.  Moreover
\begin{equation}\label{eq:bounded-overlap-Aj}
\sum_{j\ge0}\mathbf 1_{\widetilde A_j}(y)
\le
N_{\mathrm{ov}}(\Gamma)
:=1+\left\lceil \log_2\frac{2\Gamma+1}{\Gamma-1}\right\rceil.
\end{equation}
In particular, if $\Gamma\ge2$, then $N_{\mathrm{ov}}(\Gamma)\le4$.
\end{lemma}

\begin{proof}
Let $y\in A_{k,m}(x)$ and $x\in B_{r_k}(x_0)$.  Since $r_{k-m}=2^m r_k$, the triangle inequality gives
\[
|y-x_0|\ge |y-x|-|x-x_0|>\Gamma r_{k-m}-r_k
=(\Gamma-2^{-m})r_{k-m}\ge (\Gamma-1)r_{k-m},
\]
and
\[
|y-x_0|\le |y-x|+|x-x_0|\le 2\Gamma r_{k-m}+r_k
=(2\Gamma+2^{-m})r_{k-m}\le (2\Gamma+1)r_{k-m}.
\]
Thus $y\in \widetilde A_{k-m}$.  For the overlap bound, if $y\in\widetilde A_j$, then
\[
\frac{|y-x_0|}{2\Gamma+1}<r_j\le \frac{|y-x_0|}{\Gamma-1}.
\]
A dyadic sequence can meet this interval at most $1+\lceil\log_2((2\Gamma+1)/(\Gamma-1))\rceil$ times.
\end{proof}

\begin{definition}[Reassigned annular reservoirs]\label{def:reassigned-reservoirs}
For $j\le k$ set
\begin{equation}\label{eq:Ajk-def}
\mathfrak A_{j,k}
:=
\left(
 r_j^{-1}\iint_{I_k\times\widetilde A_j}|\Omega_k|^2
\right)^{1/2},
\qquad
\mathfrak A_j:=\sup_{k\ge j}\mathfrak A_{j,k}.
\end{equation}
\end{definition}

\begin{proposition}[Exact reassigned far-field bound]\label{prop:reassigned-farfield}
For every $k$,
\begin{equation}\label{eq:reassigned-farfield}
\mu_k^{\mathrm{far,ann}}
\le
C_0(\Gamma,\sigma,\chi,\varphi)
\sum_{j=0}^{k}2^{-(k-j)}\mathfrak A_{j,k}\mathcal Q_k
\le
C_0(\Gamma,\sigma,\chi,\varphi)
\sum_{j=0}^{k}2^{-(k-j)}\mathfrak A_j\mathcal Q_k.
\end{equation}
Consequently,
\begin{equation}\label{eq:shell-to-core-convolution}
\sum_{k=0}^{N}\mu_k^{\mathrm{far,ann}}
\le
C_0(\Gamma,\sigma,\chi,\varphi)
\sum_{j=0}^{N}\mathfrak A_j
\sum_{m=0}^{N-j}2^{-m}\mathcal Q_{j+m}.
\end{equation}
\end{proposition}

\begin{proof}
Start from \eqref{eq:far-shell-estimate}.  By \cref{lem:annular-reassignment}, with $j=k-m$,
\[
\mathcal E_{k,m}^{1/2}
\le
\left(r_k^{-1}\iint_{I_k\times\widetilde A_j}|\Omega_k|^2\right)^{1/2}.
\]
Since $r_k=2^{-m}r_j$, we have
\[
2^{-3m/2}
\left(r_k^{-1}\iint_{I_k\times\widetilde A_j}|\Omega_k|^2\right)^{1/2}
=
2^{-m}
\left(r_j^{-1}\iint_{I_k\times\widetilde A_j}|\Omega_k|^2\right)^{1/2}
=
2^{-m}\mathfrak A_{j,k}.
\]
This gives \eqref{eq:reassigned-farfield}.  Summing in $k$ and writing $m=k-j$ gives \eqref{eq:shell-to-core-convolution}.
\end{proof}

\subsection{Conditional unweighted Carleson closure}

\begin{theorem}[Discrete Carleson closure from reassigned annuli]\label{thm:conditional-carleson}
Assume that $\mathfrak A=(\mathfrak A_j)_{j\ge0}\in\ell^p$ and $\mathcal Q=(\mathcal Q_k)_{k\ge0}\in\ell^q$, where $1\le p,q\le\infty$ and $1/p+1/q=1$.  Then, for every $N$,
\begin{equation}\label{eq:conditional-carleson}
\sum_{k=0}^{N}\mu_k^{\mathrm{far,ann}}
\le
2C_0(\Gamma,\sigma,\chi,\varphi)
\|\mathfrak A\|_{\ell^p}\|\mathcal Q\|_{\ell^q}.
\end{equation}
In particular, the right-hand side is independent of $N$.
\end{theorem}

\begin{proof}
Let $K_m=2^{-m}\mathbf 1_{m\ge0}$.  Then $\|K\|_{\ell^1}=2$ and \eqref{eq:shell-to-core-convolution} can be written as
\[
\sum_{k=0}^{N}\mu_k^{\mathrm{far,ann}}
\le
C_0\sum_{j=0}^{N}\mathfrak A_j(K*\mathcal Q)_j.
\]
Young's inequality on sequences gives $\|K*\mathcal Q\|_{\ell^q}\le2\|\mathcal Q\|_{\ell^q}$.  H\"older's inequality then proves \eqref{eq:conditional-carleson}.
\end{proof}

\begin{corollary}[Named sufficient conditions]\label{cor:named-farfield-closures}
Each of the following assumptions implies an unweighted bound for $\sum_{k=0}^{N}\mu_k^{\mathrm{far,ann}}$ with a constant independent of $N$:
\begin{enumerate}[label=(\roman*)]
\item $\sum_j\mathfrak A_j\le A_*$ and $\sup_k\mathcal Q_k\le Q_*$;
\item $\sup_j\mathfrak A_j\le A_*$ and $\sum_k\mathcal Q_k\le Q_*$;
\item $\mathfrak A\in\ell^p$ and $\mathcal Q\in\ell^q$ with $1/p+1/q=1$;
\item $\mathfrak A_{j,k}\le \alpha_j2^{-\beta(k-j)}$ for some $\beta>0$, with $\alpha\in\ell^p$ and $\mathcal Q\in\ell^q$.
\end{enumerate}
\end{corollary}

\begin{remark}[Why reassignment alone is not enough]\label{rem:unweighted-not-from-energy}
The reassignment formula controls multiplicity in the shell variable but not summability of the coarse-shell sequence.  If the scale-invariant quantities $\mathfrak A_j$ and $\mathcal Q_k$ are merely bounded, the discrete model
\[
\mu_k\sim \sum_{j=0}^{k}2^{-(k-j)}\cdot 1\cdot 1\sim1
\]
is consistent with \eqref{eq:reassigned-farfield}, and therefore $\sum_{k=0}^N\mu_k\sim N$.  This does not assert that Navier--Stokes realizes such a model; it shows that full unweighted packing cannot be deduced from the current shell inequality alone.
\end{remark}

\subsection{Fixed-annulus harmonic route}\label{rem:fixed-annulus-harmonic-route}

The moving-shell decomposition used in the reassignment estimate is an absolute-value device.  Its shell domains depend on the core point $x$, and therefore the associated moving-shell field is not, by itself, a fixed-source harmonic field.  Harmonic Taylor expansion is legitimate only after one replaces the moving shells by a fixed smooth annular partition centered at $x_0$.

More precisely, one may choose a smooth partition of the exterior region by functions $\psi_j(y)$ supported where $|y-x_0|\simeq r_j$ and define
\[
H_{j,k}(x,t):=\int K(x-y)\psi_j(y)\Omega_k(y,t)\,dy.
\]
For $j<k$ this is a smooth exterior-source strain field on the core $B_{r_k}(x_0)$.  In an exterior-source formulation it is harmonic in the core, and after subtraction of its affine Taylor jet the remainder gains powers of $r_k/r_j$.  This gives a plausible route to affine-jet cancellation and unweighted far-field packing, but it is a separate conditional module.  No unconditional theorem in the present paper identifies the moving-shell bound of Proposition~\ref{prop:reassigned-farfield} with this fixed-source harmonic expansion.

\section{Derivative-compatible commutator forcing}\label{sec:comm}

The filtered vorticity equation contains the commutator forcing
\[
\nabla\times\nabla\!\cdot R_{\ell},
\qquad
R_\ell=S_\ell(u\otimes u)-U_\ell\otimes U_\ell.
\]
The energy-level estimate for this term loses a factor $2^{2k}$ on dyadic chains.  We replace that scale loss by a scale-invariant increment defect adapted to the differentiated filter kernel.  This is the filtered-vorticity counterpart of the increment and local-defect viewpoint in Onsager-type coarse-graining and Duchon--Robert balances \cite{ConstantinETiti1994,DuchonRobert2000,EyinkAluie2009}.  The resulting theorem identifies the genuine scale-invariant increment defect that carries any commutator-supported positive filtered-enstrophy surplus.

\subsection{Cumulants and derivative-compatible increment defects}

For $z\in\R^3$ define the velocity increment
\[
\delta_z u(x,t):=u(x-z,t)-u(x,t)
\]
and write
\[
\langle F(\cdot;x,t)\rangle_\ell
:=\int_{\R^3}\varphi_\ell(z)F(z;x,t)\,dz.
\]
Then
\[
U_\ell=u+\langle \delta u\rangle_\ell,
\]
and the exact cumulant identity is
\begin{equation}\label{eq:cumulant-R}
R_\ell
=
\langle \delta u\otimes\delta u\rangle_\ell
-
\langle\delta u\rangle_\ell\otimes\langle\delta u\rangle_\ell .
\end{equation}
This is the standard coarse-graining central moment identity in the filtering framework \cite{Germano1992,EyinkAluie2009}; it is the natural replacement for a crude estimate of $R_\ell$ by the global energy alone.

The differentiated commutator stress $\nabla\!\cdot R_\ell$ naturally carries the differentiated kernel.  Accordingly, define the probability measures
\begin{equation}\label{eq:nu-mu-def}
d\nu_\ell(z):=\varphi_\ell(z)\,dz,
\qquad
 d\mu_\ell(z):=
\frac{\ell|\nabla\varphi_\ell(z)|}{\|\nabla\varphi\|_{L^1}}\,dz.
\end{equation}
For $q\ge1$ set
\begin{equation}\label{eq:Mphi-Mgrad-def}
M_{\varphi,q}(x,t):=
\left(\int |\delta_z u(x,t)|^q\,d\nu_\ell(z)\right)^{1/q},
\qquad
M_{\nabla,q}(x,t):=
\left(\int |\delta_z u(x,t)|^q\,d\mu_\ell(z)\right)^{1/q},
\end{equation}
and define the derivative-compatible increment envelope
\begin{equation}\label{eq:M-envelope-def}
\mathfrak M_{\ell,q}(x,t):=M_{\varphi,q}(x,t)+M_{\nabla,q}(x,t).
\end{equation}
For $p\in[2,4]$ define
\begin{equation}\label{eq:S-p-def}
\widetilde{\mathcal S}_{r,\ell}^{(p)}[\chi_r]
:=
\frac r{\ell^2}
\iint \chi_r(x,t)\,\mathfrak M_{\ell,p}(x,t)^4\,dx\,dt.
\end{equation}
At fixed relative filter length $\ell=\sigma r$, this quantity is invariant under the Navier--Stokes scaling.  It is the scale-invariant commutator defect used below.  For the dyadic chain we write
\begin{equation}\label{eq:Skp-def}
\widetilde{\mathcal S}_k^{(p)}
:=\widetilde{\mathcal S}_{r_k,\ell_k}^{(p)}[\chi_k],
\qquad
\ell_k=\sigma r_k.
\end{equation}
We also use the shell version $\widetilde{\mathcal S}_{k,\mathrm{sh}}^{(p)}$, defined by replacing $\chi_k$ in \eqref{eq:S-p-def} by a cutoff supported on $\supp\nabla\chi_k$ and equal to one there.

\begin{remark}[Relation with the simpler structure-function observable]\label{rem:simple-structure-function}
The simpler quantity
\[
\mathcal S_{r,\ell}^{(p)}[\chi_r]
=
\frac r{\ell^2}\iint\chi_r
\left(\int\varphi_\ell(z)|\delta_z u|^p\,dz\right)^{4/p}dx\,dt
\]
is recovered from \eqref{eq:S-p-def} under the kernel-dominance hypothesis
\[
 d\mu_\ell\le C_{\nabla/\varphi}\,d\nu_\ell
 \qquad\hbox{for all }\ell>0.
\]
Without such a comparability assumption, $\nabla\!\cdot R_\ell$ should be estimated by the derivative-compatible envelope \eqref{eq:M-envelope-def}, because differentiating the commutator stress differentiates the filter.
\end{remark}

\subsection{An exact increment estimate for \texorpdfstring{$\nabla\!\cdot R_\ell$}{div Rl}}

\begin{lemma}[Exact increment formula and derivative-compatible control]\label{lem:increment-divR}
Let $\varphi\in C_c^1(\R^3)$ be nonnegative with unit mass, and use the convention
\[
U_\ell(x,t)=\int_{\R^3}\varphi_\ell(z)u(x-z,t)\,dz,
\qquad
\delta_z u(x,t)=u(x-z,t)-u(x,t).
\]
Set
\[
m_\ell(x,t):=\int_{\R^3}\varphi_\ell(z)\delta_z u(x,t)\,dz=U_\ell(x,t)-u(x,t).
\]
Then, for each component $i=1,2,3$,
\begin{equation}\label{eq:divR-increment-identity}
(\nabla\!\cdot R_\ell)_i
=
\int_{\R^3}(\partial_j\varphi_\ell)(z)\,\delta_z u_i\,\delta_z u_j\,dz
-
 m_{\ell,j}\int_{\R^3}(\partial_j\varphi_\ell)(z)\,\delta_z u_i\,dz,
\end{equation}
with summation in $j$.  Consequently, for every $p\in[2,\infty)$,
\begin{equation}\label{eq:pointwise-divR-compatible}
|\nabla\!\cdot R_\ell|(x,t)
\le
\frac{C_{\mathrm{div}}\|\nabla\varphi\|_{L^1}}{\ell}
\bigl(M_{\nabla,2}^2+M_{\varphi,1}M_{\nabla,1}\bigr)
\le
\frac{C_{\mathrm{div}}\|\nabla\varphi\|_{L^1}}{\ell}\,\mathfrak M_{\ell,p}(x,t)^2,
\end{equation}
and hence
\begin{equation}\label{eq:pointwise-divR-p}
|\nabla\!\cdot R_\ell|^2(x,t)
\le
\frac{C_{\mathrm{div}}^2\|\nabla\varphi\|_{L^1}^2}{\ell^2}\,
\mathfrak M_{\ell,p}(x,t)^4.
\end{equation}
Here $C_{\mathrm{div}}$ is a universal dimensional constant.
\end{lemma}

\begin{proof}
Since
\[
R_{\ell,ij}(x,t)
=
\int\varphi_\ell(z)u_i(x-z,t)u_j(x-z,t)\,dz
-U_{\ell,i}(x,t)U_{\ell,j}(x,t),
\]
we compute
\[
(\nabla\!\cdot R_\ell)_i
=\partial_j(\varphi_\ell*(u_i u_j))-
\partial_j(U_{\ell,i}U_{\ell,j}).
\]
Because $\nabla\!\cdot U_\ell=0$,
\[
\partial_j(U_{\ell,i}U_{\ell,j})=(\partial_jU_{\ell,i})U_{\ell,j}.
\]
Using the convention above and integrating by parts in $z$,
\[
\partial_j(\varphi_\ell*(u_i u_j))(x,t)
=
\int(\partial_j\varphi_\ell)(z)u_i(x-z,t)u_j(x-z,t)\,dz,
\]
and
\[
\partial_jU_{\ell,i}(x,t)
=
\int(\partial_j\varphi_\ell)(z)u_i(x-z,t)\,dz.
\]
Hence
\[
(\nabla\!\cdot R_\ell)_i
=
\int(\partial_j\varphi_\ell)(z)u_i(x-z)\bigl(u_j(x-z)-U_{\ell,j}(x)\bigr)\,dz.
\]
Write
\[
u_i(x-z)=u_i(x)+\delta_z u_i(x),
\qquad
u_j(x-z)-U_{\ell,j}(x)=\delta_z u_j(x)-m_{\ell,j}(x).
\]
The term proportional to $u_i(x)$ vanishes because
\[
\int(\partial_j\varphi_\ell)(z)\bigl(\delta_z u_j-m_{\ell,j}\bigr)\,dz
=
\partial_jU_{\ell,j}(x)=0.
\]
This proves \eqref{eq:divR-increment-identity}.

For the estimate, use \eqref{eq:nu-mu-def}.  The first term in \eqref{eq:divR-increment-identity} is bounded by
\[
\int |\nabla\varphi_\ell(z)|\,|\delta_z u|^2\,dz
=
\frac{\|\nabla\varphi\|_{L^1}}{\ell}M_{\nabla,2}^2,
\]
up to the universal dimensional summation constant.  The second term is bounded by
\[
|m_\ell|\int |\nabla\varphi_\ell(z)|\,|\delta_z u|\,dz
\le
M_{\varphi,1}\frac{\|\nabla\varphi\|_{L^1}}{\ell}M_{\nabla,1}.
\]
Because $d\nu_\ell$ and $d\mu_\ell$ are probability measures and $p\ge2$,
\[
M_{\varphi,1}\le M_{\varphi,p},
\qquad
M_{\nabla,1}\le M_{\nabla,p},
\qquad
M_{\nabla,2}\le M_{\nabla,p}.
\]
Thus the right-hand side is controlled by $\ell^{-1}\|\nabla\varphi\|_{L^1}\mathfrak M_{\ell,p}^2$, proving \eqref{eq:pointwise-divR-compatible}.  Squaring gives \eqref{eq:pointwise-divR-p}.
\end{proof}

\subsection{Derivative-compatible commutator insertion}

Recall the commutator forcing
\begin{equation}\label{eq:Fcomm-def-repeat}
F_k^{\mathrm{com}}
:=
r_k\left|
\iint_{Q_k}\chi_k\,
\Omega_k\cdot(\nabla\times\nabla\!\cdot R_k)
\right|.
\end{equation}

\begin{theorem}[Derivative-compatible commutator insertion]\label{thm:increment-comm}
For every $p\in[2,4]$ and every $\eta\in(0,1)$,
\begin{equation}\label{eq:increment-comm}
F_k^{\mathrm{com}}
\le
\eta P_k
+
\frac{C_{\mathrm{com}}^{\sharp}(\varphi)}{\eta}\,\widetilde{\mathcal S}_k^{(p)}
+
L_{k,\mathrm{inc}}^{\mathrm{com}},
\end{equation}
where one may take
\begin{equation}\label{eq:Ccomm-def}
C_{\mathrm{com}}^{\sharp}(\varphi)
=C_*\|\nabla\varphi\|_{L^1}^2
\end{equation}
with a universal numerical constant $C_*$.  The shell budget may be chosen as
\begin{equation}\label{eq:Lcomm-inc-def}
L_{k,\mathrm{inc}}^{\mathrm{com}}
:=
C_\chi r_k^{-1}\iint_{\supp\nabla\chi_k}|\Omega_k|^2
+
C_\chi r_k\iint_{\supp\nabla\chi_k}|\nabla\!\cdot R_k|^2.
\end{equation}
Moreover,
\begin{equation}\label{eq:shell-increment-control}
r_k\iint_{\supp\nabla\chi_k}|\nabla\!\cdot R_k|^2
\le
C_{\mathrm{com}}^{\sharp}(\varphi)\,\widetilde{\mathcal S}_{k,\mathrm{sh}}^{(p)}.
\end{equation}
Thus the commutator localization residual can be written as an annular filtered-enstrophy term plus an annular derivative-compatible increment defect.
\end{theorem}

\begin{proof}
Use
\[
\int \chi a\cdot(\nabla\times b)
=
\int \chi b\cdot(\nabla\times a)
+
\int (\nabla\chi\times a)\cdot b.
\]
With $a=\Omega_k$ and $b=\nabla\!\cdot R_k$, this gives
\begin{align*}
F_k^{\mathrm{com}}
\le{}&
r_k\iint\chi_k|\nabla\!\cdot R_k|\,|\nabla\times\Omega_k|\,dxdt\\
&+
r_k\iint|\nabla\chi_k|\,|\Omega_k|\,|\nabla\!\cdot R_k|\,dxdt .
\end{align*}
Since $|\nabla\times\Omega_k|\le C|\nabla\Omega_k|$, Young's inequality yields
\[
r_k\iint\chi_k|\nabla\!\cdot R_k|\,|\nabla\times\Omega_k|
\le
\eta P_k
+
\frac{C}{\eta}r_k\iint\chi_k|\nabla\!\cdot R_k|^2.
\]
By \cref{lem:increment-divR}, the last integral is bounded by $C_{\mathrm{com}}^{\sharp}(\varphi)\widetilde{\mathcal S}_k^{(p)}$.  The cutoff term is estimated by Young and the derivative bound $|\nabla\chi_k|\le C_\chi r_k^{-1}$, which gives \eqref{eq:Lcomm-inc-def}.  Applying \cref{lem:increment-divR} on the shell gives \eqref{eq:shell-increment-control}.
\end{proof}

\begin{remark}[Energy-level fallback]\label{rem:energy-fallback}
If the increment observable is discarded and only the global energy is used, the crude estimate
\[
\|\nabla\!\cdot R_{\ell_k}(t)\|_{L^2}
\lesssim
\ell_k^{-5/2}\|u(t)\|_{L^2}^2
\]
gives $Q_{k,\mathrm{div}}^{\mathrm{com}}\lesssim M_E^2 2^{2k}$ when $M_E=\|u\|_{L_t^\infty L_x^2}^2$.  This recovers the older $\mathcal W_2$ loss.  The derivative-compatible theorem above keeps the actual scale-invariant increment defect instead of replacing it by that scale-worse energy bound.
\end{remark}

\begin{corollary}[Weighted commutator alternative]\label{cor:weighted-comm-increment}
For every nonnegative weight sequence $w$ and every $p\in[2,4]$,
\begin{equation}\label{eq:weighted-comm-increment}
\sum_{k=0}^N w_k F_k^{\mathrm{com}}
\le
\eta\sum_{k=0}^N w_k P_k
+
C_{\eta,\varphi}\sum_{k=0}^N w_k\widetilde{\mathcal S}_k^{(p)}
+
\sum_{k=0}^N w_k L_{k,\mathrm{inc}}^{\mathrm{com}}.
\end{equation}
\end{corollary}

\begin{proof}
Sum \eqref{eq:increment-comm} over $k$ with weights $w_k$.
\end{proof}

\section{Weighted surplus inequality}\label{sec:weighted-surplus}

We now combine the near-field coercivity, far-field packing, and increment-based commutator insertion.

For each scale $k$, let $E_{k,\mathrm{in}}^\omega$ and $E_{k,\mathrm{out}}^\omega$ denote the endpoint terms in the filtered enstrophy identity on the corresponding time slab.  Split the diffusion budget as
\[
\eta=\eta_{\mathrm{near}}+\eta_{\mathrm{com}},
\qquad
\eta_{\mathrm{near}},\eta_{\mathrm{com}}>0,
\qquad
\eta_{\mathrm{near}}+\eta_{\mathrm{com}}<1.
\]
We define the positive post-near-field defect surplus by
\begin{equation}\label{eq:Sk-def}
\mathfrak S_k
:=
\left[
E_{k,\mathrm{out}}^\omega
+
(1-\eta_{\mathrm{near}}-\eta_{\mathrm{com}})P_k
-
E_{k,\mathrm{in}}^\omega
-
C_{\eta_{\mathrm{near}},\sigma,\rho,\varphi,\chi}M_k\mathcal O_k
\right]_+ .
\end{equation}
The constant is chosen large enough to absorb the lower-order filtered-enstrophy term produced by \cref{cor:fixed-ratio}.  Thus $\mathfrak S_k$ is a defect-balance quantity, not a pointwise component of the stretching tensor.

\begin{theorem}[Weighted surplus inequality for the filtered enstrophy balance]\label{thm:weighted-surplus}
Let $u$ be a Leray--Hopf solution on $\R^3$.  Fix $0<\sigma\le\rho\le1/4$, $r_k=2^{-k}$, $\ell_k=\sigma r_k$, and let $p\in[2,4]$.  Then for every finite $N$ and every $w\in\mathcal W_{3/2}$,
\begin{align}
\sum_{k=0}^{N}w_k\mathfrak S_k
\le{}&
C_{\rho,\sigma,\chi,\varphi}M_E^{3/2}
\|w\|_{\mathcal W_{3/2}}
+
C_{\eta_{\mathrm{com}},\varphi}\sum_{k=0}^{N}w_k\widetilde{\mathcal S}_k^{(p)}
\nonumber\\
&+
\sum_{k=0}^{N}w_k\bigl(L_k+L_{k,\mathrm{inc}}^{\mathrm{com}}\bigr).
\label{eq:weighted-surplus}
\end{align}
\end{theorem}

\begin{proof}[Proof of \cref{thm:weighted-surplus}]
The localized filtered enstrophy identity at scale $k$, after taking positive parts of the remaining signed work terms, gives
\[
E_{k,\mathrm{out}}^\omega+P_k
\le
E_{k,\mathrm{in}}^\omega
+V_k^{+,\mathrm{near}}
+V_k^{+,\mathrm{far}}
+F_k^{\mathrm{com}}
+L_k.
\]
The near-field coercivity at fixed relative filter length gives
\[
V_k^{+,\mathrm{near}}
\le
\eta_{\mathrm{near}}P_k
+C_{\eta_{\mathrm{near}},\sigma,\rho,\varphi,\chi}M_k\mathcal O_k
+L_k.
\]
Therefore, after increasing the harmless constant in \eqref{eq:Sk-def},
\[
\mathfrak S_k
\le
V_k^{+,\mathrm{far}}
+F_k^{\mathrm{com}}
-\eta_{\mathrm{com}}P_k
+L_k.
\]
Applying \cref{thm:increment-comm} with parameter $\eta_{\mathrm{com}}$ gives
\[
F_k^{\mathrm{com}}-\eta_{\mathrm{com}}P_k
\le
C_{\eta_{\mathrm{com}},\varphi}\widetilde{\mathcal S}_k^{(p)}
+L_{k,\mathrm{inc}}^{\mathrm{com}}.
\]
Thus
\[
\mathfrak S_k
\le
V_k^{+,\mathrm{far}}
+C_{\eta_{\mathrm{com}},\varphi}\widetilde{\mathcal S}_k^{(p)}
+L_k+L_{k,\mathrm{inc}}^{\mathrm{com}}.
\]
Multiplying by $w_k$, summing over $k$, and using \cref{thm:farfield-packing} proves \eqref{eq:weighted-surplus}.
\end{proof}

\begin{corollary}[Conditional unweighted surplus bound]\label{cor:conditional-unweighted-surplus}
Assume, in addition to the hypotheses of \cref{thm:weighted-surplus}, that the full far-field contribution satisfies an unweighted Carleson closure of the form
\[
\sum_{k=0}^{N}V_k^{+,\mathrm{far}}
\le
\mathcal F_*
\qquad\text{with }\mathcal F_*\text{ independent of }N.
\]
For example, the reassigned annular part satisfies such a bound under the hypotheses of \cref{thm:conditional-carleson}, and any exterior tail should be controlled separately.  Then
\begin{align*}
\sum_{k=0}^{N}\mathfrak S_k
\le{}&
\mathcal F_*
+C_{\eta_{\mathrm{com}},\varphi}\sum_{k=0}^{N}\widetilde{\mathcal S}_k^{(p)}
+\sum_{k=0}^{N}\bigl(L_k+L_{k,\mathrm{inc}}^{\mathrm{com}}\bigr).
\end{align*}
\end{corollary}

\begin{proof}
Repeat the proof of \cref{thm:weighted-surplus} with $w_k\equiv1$ and replace the weighted far-field packing estimate by the stated unweighted far-field closure.
\end{proof}

\begin{theorem}[Conditional unweighted closure and vanishing defect surplus]\label{thm:balance-closure-vanishing}
Work in the exact whole-space setting.  Assume that the principal cutoff residual is eliminated by an adjoint cutoff or is included in a summable nonnegative shell budget.  Suppose that
\begin{enumerate}[label=\textup{(\roman*)},leftmargin=2.2em]
\item the full far-field contribution is unweightedly summable, for instance by \cref{thm:conditional-carleson} together with a separate exterior-tail estimate;
\item for some $p\in[2,4]$,
\[
\sum_{k=0}^{\infty}\widetilde{\mathcal S}_k^{(p)}<\infty;
\]
\item the remaining shell budgets are summable:
\[
\sum_{k=0}^{\infty}\bigl(L_k+L_{k,\mathrm{inc}}^{\mathrm{com}}\bigr)<\infty.
\]
\end{enumerate}
Then
\[
\sum_{k=0}^{\infty}\mathfrak S_k<\infty,
\qquad
\mathfrak S_k\to0.
\]
Thus persistent positive post-near-field defect surplus cannot survive at arbitrarily small scales under these closure hypotheses.
\end{theorem}

\begin{proof}
The proof of \cref{thm:weighted-surplus} gives the pointwise-in-scale bound
\[
\mathfrak S_k
\le
V_k^{+,\mathrm{far}}
+C_{\eta_{\mathrm{com}},\varphi}\widetilde{\mathcal S}_k^{(p)}
+L_k+L_{k,\mathrm{inc}}^{\mathrm{com}}.
\]
Summing over $k=0,\dots,N$ and using the three hypotheses gives a bound independent of $N$ for $\sum_{k=0}^{N}\mathfrak S_k$.  Since $\mathfrak S_k\ge0$, the infinite sum is finite and hence $\mathfrak S_k\to0$.
\end{proof}

\begin{corollary}[Persistent defect-surplus alternative]\label{cor:filtered-increment-defect-alt}
For every $p\in[2,4]$,
\begin{equation}\label{eq:filtered-increment-defect-lower}
\sum_{k=0}^{N}w_k\widetilde{\mathcal S}_k^{(p)}
\ge
\frac1{C_{\eta_{\mathrm{com}},\varphi}}
\left[
\sum_{k=0}^{N}w_k\mathfrak S_k
-
C_{\rho,\sigma,\chi,\varphi}M_E^{3/2}\|w\|_{\mathcal W_{3/2}}
-
\sum_{k=0}^{N}w_k(L_k+L_{k,\mathrm{inc}}^{\mathrm{com}})
\right]_+.
\end{equation}
Consequently, if the weighted post-near-field defect surplus persists after the far-field and localization budgets are removed, then a weighted increment defect persists.  This statement concerns the localized filtered enstrophy balance: the commutator term is not a component of the pointwise stretching tensor, but the forcing term that should pay any remaining positive defect surplus.
\end{corollary}

\begin{remark}[From energy loss to scale-invariant increment defect]\label{rem:compare-increment-energy}
The increment-based theorem removes the forced scale-worse energy loss and replaces it by the genuine scale-invariant increment defect $\sum w_k\widetilde{\mathcal S}_k^{(p)}$.  Thus commutator-supported defect surplus is accompanied by an explicit derivative-compatible velocity-increment defect, which can subsequently be analyzed through compactness and rigidity.
\end{remark}

\section{Obstruction-profile interpretation}\label{sec:obstruction}

The preceding theorem is a finite-chain inequality.  It also has a compactness interpretation.  The quantity $\mathbb S_\ell\Omega_\ell\cdot\Omega_\ell$ is the stretching work, while the commutator term enters only through the localized filtered enstrophy balance.  Thus the assertion is not that the commutator term is a pointwise part of the stretching tensor.  Rather, after the near-field term has been absorbed, any positive defect surplus not paid by far-field work or localization should be paid by the differentiated commutator forcing.  The increment estimate of \cref{thm:increment-comm} converts this forcing into the scale-critical increment defect.

\subsection{Defect surplus forces the commutator increment defect}

We first record the precise algebraic form of this statement.  Split the diffusion budget as
\[
\eta=\eta_{\mathrm{near}}+\eta_{\mathrm{com}},
\qquad
\eta_{\mathrm{near}},\eta_{\mathrm{com}}>0.
\]
For a normalized sequence at a fixed relative filter length $\ell=\sigma r$, let
\[
E_{n,\mathrm{in}}^\omega,
\quad
E_{n,\mathrm{out}}^\omega,
\quad
P_n,
\quad
\mathcal O_n,
\quad
F_n^{\mathrm{far}},
\quad
L_n,
\quad
L_n^{\mathrm{com}}
\]
denote the corresponding endpoint enstrophy, diffusion, reservoir, far-field work, localization residual, and commutator shell residual terms after rescaling to a fixed unit cylinder.  Define the post-near-field defect surplus
\begin{equation}\label{eq:defect-surplus-B}
\mathfrak B_n
:=
\left[
E_{n,\mathrm{out}}^\omega
+
(1-\eta_{\mathrm{near}}-\eta_{\mathrm{com}})P_n
-
E_{n,\mathrm{in}}^\omega
-
C_{\eta,\sigma}M\mathcal O_n
-
F_n^{\mathrm{far}}
-
L_n
-
L_n^{\mathrm{com}}
\right]_+ .
\end{equation}

\begin{proposition}[Defect surplus forces the commutator increment defect]\label{prop:balance-to-comm-young}
Assume the near-field estimate holds with parameter $\eta_{\mathrm{near}}$ and the commutator insertion \cref{thm:increment-comm} holds with parameter $\eta_{\mathrm{com}}$ and $p=3$.  Set $C_{\eta_{\mathrm{com}},\varphi}^{\sharp}:=C_{\mathrm{com}}^{\sharp}(\varphi)/\eta_{\mathrm{com}}$.  If a normalized sequence satisfies
\[
\mathfrak B_n\ge s_0>0,
\]
then
\begin{equation}\label{eq:balance-to-comm-liminf}
\liminf_{n\to\infty}\widetilde{\mathcal S}_n^{(3)}
\ge
\frac{s_0}{C_{\eta_{\mathrm{com}},\varphi}^{\sharp}}
=
\frac{\eta_{\mathrm{com}}}{C_{\mathrm{com}}^{\sharp}(\varphi)}\,s_0.
\end{equation}
Equivalently, using \eqref{eq:Ccomm-def},
\begin{equation}\label{eq:balance-to-comm-explicit}
\liminf_{n\to\infty}\widetilde{\mathcal S}_n^{(3)}
\ge
\frac{\eta_{\mathrm{com}}}{C_{\mathrm{com}}^{\sharp}(\varphi)}\,s_0.
\end{equation}
In particular, after the far-field, principal localization residual, and commutator shell residual have been eliminated or shown to be negligible, any persistent positive defect surplus forces a persistent commutator increment defect.
\end{proposition}

\begin{proof}
The localized filtered enstrophy identity gives, 
\[
E_{n,\mathrm{out}}^\omega+P_n
\le
E_{n,\mathrm{in}}^\omega
+V_n^{+,\mathrm{near}}
+F_n^{\mathrm{far}}
+F_n^{\mathrm{com}}
+L_n.
\]
The near-field coercivity with parameter $\eta_{\mathrm{near}}$ gives
\[
V_n^{+,\mathrm{near}}
\le
\eta_{\mathrm{near}}P_n
+
C_{\eta,\sigma}M\mathcal O_n.
\]
Hence
\[
E_{n,\mathrm{out}}^\omega
+(1-\eta_{\mathrm{near}})P_n
-
E_{n,\mathrm{in}}^\omega
-C_{\eta,\sigma}M\mathcal O_n
-F_n^{\mathrm{far}}
-L_n
\le
F_n^{\mathrm{com}}.
\]
Using \cref{thm:increment-comm} with parameter $\eta_{\mathrm{com}}$ and $p=3$,
\[
F_n^{\mathrm{com}}
\le
\eta_{\mathrm{com}}P_n
+
C_{\eta_{\mathrm{com}},\varphi}^{\sharp}\widetilde{\mathcal S}_n^{(3)}
+
L_n^{\mathrm{com}}.
\]
Moving $\eta_{\mathrm{com}}P_n$ and $L_n^{\mathrm{com}}$ to the left gives
\[
\mathfrak B_n
\le
C_{\eta_{\mathrm{com}},\varphi}^{\sharp}\widetilde{\mathcal S}_n^{(3)}.
\]
The lower bound \eqref{eq:balance-to-comm-liminf} follows immediately.
\end{proof}

\begin{remark}[The role of the commutator term]\label{rem:balance-not-pointwise-stretching}
The commutator term enters the filtered vorticity equation as a forcing term.  The balance formulation captures this distinction exactly: after near-field absorption and removal of the far-field and localization budgets, any remaining positive balance is paid by commutator forcing.  The increment estimate then identifies the scale-critical object carrying that payment.
\end{remark}

\subsection{Young profiles for derivative-compatible increment defects}

We now describe a conservative compactness object associated with a bounded nonvanishing $p=3$ derivative-compatible increment defect.  Fix a filter ratio $\sigma>0$ and let $c_\varphi$ be such that $\supp\varphi\subset B_{c_\varphi}$.  At unit scale define
\[
d\nu_\sigma(z):=\varphi_\sigma(z)\,dz,
\qquad
 d\mu_\sigma(z):=\frac{\sigma|\nabla\varphi_\sigma(z)|}{\|\nabla\varphi\|_{L^1}}\,dz.
\]
The natural increment space is
\begin{equation}\label{eq:Esharp-def}
E^\sharp_\sigma
:=
L^3(B_{c_\varphi\sigma},d\nu_\sigma;\R^3)
\times
L^3(B_{c_\varphi\sigma},d\mu_\sigma;\R^3),
\end{equation}
endowed with the sum norm.  This is a separable reflexive Banach space.  For a normalized sequence $u^{(n)}$, set
\begin{equation}\label{eq:Vsharp-def}
V^\sharp_n(x,t)(z)
:=
\bigl(\delta_z u^{(n)}(x,t),\delta_z u^{(n)}(x,t)\bigr)
\in E^\sharp_\sigma,
\qquad
\delta_z u^{(n)}(x,t)=u^{(n)}(x-z,t)-u^{(n)}(x,t).
\end{equation}
Then, at unit scale,
\begin{equation}\label{eq:S-Vsharp}
\widetilde{\mathcal S}^{(3)}_n
=
\sigma^{-2}\iint\chi(x,t)\|V^\sharp_n(x,t)\|_{E^\sharp_\sigma}^{4}\,dxdt.
\end{equation}

\begin{theorem}[Cylindrical commutator increment profile]\label{thm:young-extraction}
Let $Q^+$ be a fixed parabolic cylinder containing all spatial shifts of $Q$ by vectors in $B_{c_\varphi\sigma}$.  Suppose $u^{(n)}$ is bounded in $L^3(Q^+)$ and
\[
\sup_n\widetilde{\mathcal S}^{(3)}_n<\infty.
\]
Then, after passing to a subsequence, the derivative-compatible increment fields $V^\sharp_n$ generate a cylindrical generalized Young profile on $E^\sharp_\sigma$.  More precisely, for every finite family of continuous linear functionals $\ell_1,\dots,\ell_m\in(E^\sharp_\sigma)^*$, the projected maps
\[
(\ell_1(V^\sharp_n),\dots,\ell_m(V^\sharp_n))
\]
generate a finite-dimensional DiPerna--Majda generalized Young measure, and these projected measures are consistent under further projections.  If $u^{(n)}\rightharpoonup u$ weakly in $L^3(Q^+)$, then the cylindrical barycenter is the resolved derivative-compatible increment field
\[
V^\sharp(x,t)(z)=\bigl(u(x-z,t)-u(x,t),u(x-z,t)-u(x,t)\bigr).
\]
\end{theorem}

\begin{proof}
The bound on $\widetilde{\mathcal S}^{(3)}_n$ is exactly a bound on $\chi^{1/4}V^\sharp_n$ in the reflexive space $L^4(Q;E^\sharp_\sigma)$.  For any finite family of linear functionals on $E^\sharp_\sigma$, the projected sequence is bounded in a finite-dimensional $L^4$ space.  The finite-dimensional DiPerna--Majda compactness theorem yields a generalized Young measure for that projection.  A diagonal extraction over a countable separating family of linear functionals gives a projectively consistent cylindrical profile.

If $u^{(n)}\rightharpoonup u$ in $L^3(Q^+)$, the increment operator $u\mapsto u(\cdot-z,\cdot)-u(\cdot,\cdot)$ is bounded from $L^3(Q^+)$ into $L^3(Q;E^\sharp_\sigma)$.  Hence the weak barycenter of every projected Young measure is the projection of the resolved increment field.  This identifies the cylindrical barycenter.
\end{proof}

\begin{remark}[Cylindrical versus full Young measures]\label{rem:cylindrical-vs-full}
A bounded sequence in an infinite-dimensional Banach space need not be tight in the norm topology merely because all finite-dimensional projections are controlled.  Therefore the cylindrical formulation is the unconditional compactness statement proved here.  Statements involving non-cylindrical functionals, such as the full norm $\|\Xi\|_{E^\sharp_\sigma}^{4}$ or the commutator covariance map, require an additional full-representation hypothesis.
\end{remark}

\subsection{Conditional full-profile consequences}

\begin{definition}[Full increment Young representation]\label{def:full-ym-assumption}
We say that $V^\sharp_n$ admits a full generalized Young representation on $E^\sharp_\sigma$ if, after subsequence extraction, there exist a generalized Young measure $(\nu_{x,t},\lambda,\nu^\infty_{x,t})$ and the standard representation formulae hold for the full norm functional
\[
G(\Xi)=\|\Xi\|_{E^\sharp_\sigma}^{4}
\]
and for the quadratic commutator covariance map below.  This is an additional compactness hypothesis, not a consequence of \cref{thm:young-extraction} alone.
\end{definition}

\begin{proposition}[Unresolved excess and commutator stress defects under full representation]\label{prop:comm-stress-defect}
Assume the hypotheses of \cref{thm:young-extraction} and the full representation property of \cref{def:full-ym-assumption}.  Then
\[
\liminf_{n\to\infty}\widetilde{\mathcal S}^{(3)}_n
\ge
\widetilde{\mathcal S}^{(3)}[u]+\mathcal D^{(3)}_\sigma,
\]
where $\mathcal D^{(3)}_\sigma\ge0$ is the Jensen/concentration gap.  If the left-hand side is strictly larger than $\widetilde{\mathcal S}^{(3)}[u]$, then either the oscillation measure is non-Dirac on a set of positive measure or the concentration measure is nonzero.

Define the covariance map, using the $d\nu_\sigma$ component, by
\begin{equation}\label{eq:covariance-map-safe}
\mathcal C(\Xi)
:=
\int\varphi_\sigma(z)\Xi_\nu(z)\otimes\Xi_\nu(z)\,dz
-
\left(\int\varphi_\sigma(z)\Xi_\nu(z)\,dz\right)^{\otimes2}.
\end{equation}
It has quadratic growth.  Hence, under the full representation hypothesis,
\[
R_\sigma[u^{(n)}]\rightharpoonup R^{\mathrm{YM}}_\sigma:=\langle\nu_{x,t},\mathcal C\rangle
\qquad\text{weakly in }L^2_{\mathrm{loc}}.
\]
If
\[
D:=R^{\mathrm{YM}}_\sigma-R_\sigma[u]\ne0,
\]
then the full increment profile is nontrivial.  If the concentration measure vanishes, then the oscillation measure is non-Dirac on a set of positive measure.
\end{proposition}

\begin{proof}
The first assertion follows from the generalized Young-measure lower semicontinuity theorem applied to the convex functional $G(\Xi)=\|\Xi\|_{E^\sharp_\sigma}^{4}$ and from Jensen's inequality.  Strict excess should be carried either by strict Jensen inequality, which means non-Dirac oscillation, or by the concentration measure.  The covariance statement follows because $|\mathcal C(\Xi)|\lesssim \|\Xi\|_{E^\sharp_\sigma}^2$.  Since this growth is subcritical relative to the quartic bound, the concentration part does not contribute to the weak $L^2$ limit of the covariance map.  A Dirac, concentration-free profile gives $R^{\mathrm{YM}}_\sigma=R_\sigma[u]$, so a nonzero defect forces nontrivial microstructure.
\end{proof}

\begin{remark}[No converse from vanishing stress defect]\label{rem:no-converse-stress-defect}
The implication above is one-way.  A nonzero commutator stress defect forces nontrivial increment microstructure, but a zero commutator stress defect does not force the Young profile to be Dirac.  The covariance map is not injective.  Therefore this paper does not use the false implication $D=0\Rightarrow\nu_{x,t}$ is Dirac.
\end{remark}

\subsection{Defect work and the commutator recurrence test}

The extracted Young measure separates the scalar size of the commutator increment defect from the microstructure that carries it.  The next question is whether this microstructure can generate positive recurrent commutator work.  Formally, once a defect stress
\[
R_\sigma^{\mathrm{def}}=R_\sigma^{\mathrm{YM}}-R_\sigma[u]
\]
has been extracted, one may test the associated defect forcing by
\begin{equation}\label{eq:comm-defect-work}
W_{\mathrm{com}}^{\mathrm{def}}
:=
r\left|
\iint \chi_r\,
\Omega_\ell\cdot \nabla\times\nabla\!\cdot R_\ell^{\mathrm{def}}\,dxdt
\right|,
\end{equation}
under any mollified or distributionally justified interpretation of the pairing.  The corresponding defect increment mass is
\begin{equation}\label{eq:S-def-3}
S_{\mathrm{def}}^{(3)}
:=
\frac r{\ell^2}
\iint \chi_r
\left\langle \nu_{x,t},
\|W-\bar V(x,t)\|_{E_\sigma^\sharp}^4
\right\rangle dxdt,
\end{equation}
where $\bar V(x,t)=\langle\nu_{x,t},W\rangle$ is the resolved increment.  A natural scale-by-scale defect-work ratio is therefore
\begin{equation}\label{eq:comm-efficiency}
\mathfrak E_{\mathrm{com}}
:=
\frac{(W_{\mathrm{com}}^{\mathrm{def}})_+}{S_{\mathrm{def}}^{(3)}+\varepsilon},
\qquad \varepsilon>0.
\end{equation}
The ratio $\mathfrak E_{\mathrm{com},k}$ organizes the next closure test: whether a fixed portion of the defect increment mass can be converted into positive commutator work along a dyadic chain.  Decay of this defect-work ratio would render a persistent increment defect harmless for positive defect work, whereas a positive lower bound would identify a recurrent commutator defect.

\begin{remark}[Young-measure profile as a recurrence test]\label{rem:young-not-positive-defect-work}
The extraction theorem supplies the compactness object on which cross-scale positive defect work can be formulated quantitatively.  The decisive question is whether a nontrivial increment Young measure can sustain positive recurrent commutator work from one scale to the next.  A vanishing defect-work ratio would rigidify the commutator term; a persistent defect-work ratio would isolate a candidate persistent positive defect work in any counterexample scenario.  Thus the Young-measure profile converts an open qualitative question into the explicit recurrence test \eqref{eq:comm-efficiency}.
\end{remark}

\section{Consequences and remaining alternatives}\label{sec:discussion}

The estimates above give a closed near-field coercive bound and leave three explicit residual classes: far-field terms, differentiated commutator terms, and localization residuals.  The core implication is
\[
\begin{aligned}
\text{near-field positive stretching}
&\Longrightarrow \text{filtered direction incoherence}\\
&\Longrightarrow \text{vorticity difference quotients}\\
&\Longrightarrow \text{filtered diffusion plus lower-order enstrophy}.
\end{aligned}
\]
At one scale this is \cref{cor:stretching-coercivity}; on a finite dyadic chain it becomes the surplus inequality \cref{thm:weighted-surplus}.  The significance of the balance formulation is that any positive surplus surviving the near-field absorption is no longer an undifferentiated nonlinear remainder.  It is assigned to a finite collection of analytically structured residual terms.

\subsection{Near-field coercivity}

The singular part of the stretching interaction satisfies the scale-normalized chain
\[
\mathcal V^{+,\mathrm{near}}_{r,\ell}
\lesssim
\mathcal A^{\mathrm{pair}}_{r,\ell}
\lesssim
\eta\mathcal P^\rho_{r,\ell}
+
C_\eta M_{r,\rho}(u)
\left(\frac r\ell\right)^5\mathcal O_{r,\ell}.
\]
This closes the near-field term at fixed relative filter length.  The angular cancellation is exact at the filtered level, and the magnitude weights neutralize the zero set of $\Omega_\ell$ without introducing denominator regularization.

\subsection{Far-field annular packing}

The weighted factor in \cref{thm:farfield-packing} is the baseline cost of estimating the nonlocal strain tail from energy alone.  Proposition~\ref{prop:reassigned-farfield} gives a moving-shell reassignment: it rewrites part of the far-field work as a discrete convolution between annular vorticity reservoirs and core enstrophy profiles.  This proves a conditional Carleson closure theorem under explicit summability assumptions on the reassigned reservoirs and core profiles.

A harmonic Taylor expansion requires fixed exterior sources.  Accordingly, this paper treats harmonic affine-jet cancellation as a separate conditional route, described in \cref{rem:fixed-annulus-harmonic-route}.  A complete harmonic-rigidity theorem should begin from a fixed smooth annular partition centered at the base point and control the resulting affine jets directly.

\subsection{Commutator increment microstructure}

The commutator forcing is differentiated, and its natural observable should therefore see both the filter and its derivative.  The envelope $\mathfrak M_{\ell,p}$ and the scale-invariant quantity $\widetilde{\mathcal S}^{(p)}_{r,\ell}$ retain this structure exactly.  \Cref{thm:increment-comm} consequently replaces the dyadic energy loss by an intrinsic derivative-compatible scale-invariant increment defect.

A scalar increment bound records the size of this defect but not, by itself, the microstructure carrying it.  The unconditional compactness statement in this paper is a cylindrical generalized Young profile for the derivative-compatible increment field.  Any further use of the full norm functional or commutator covariance map is stated under the explicit full-representation hypothesis in \cref{def:full-ym-assumption}.  Under that additional hypothesis, unresolved quartic excess yields oscillation or concentration, and nonzero commutator stress defect forces a nontrivial increment profile.  The converse is not asserted.

\subsection{Localization as a modular budget}

In the exact whole-space setting, the principal cutoff residual is cancelled by the adjoint drift-diffusion weight of \cref{prop:adjoint-residual}, provided one allows the corresponding solution-adapted weight.  The remaining shell terms are explicit nonnegative budgets generated by enlarged diffusion regions and integration by parts in the commutator term.  A compactly supported local-cylinder theorem would add a solenoidal extension or localized filter and its transition budgets.  Those terms form a separate localization module.

\subsection{Recurrence versus rigidity}

The set of alternatives now has three precise inputs: unweighted far-field packing, summability or rigidity of the derivative-compatible commutator increment defect, and summability of the residual localization budgets.  Under these inputs, \cref{thm:balance-closure-vanishing} yields decay of the positive defect surplus.

The remaining commutator question is sharper than the existence of a nonzero scalar increment bound.  It is whether the extracted oscillation-concentration profile can repeatedly convert unresolved increment mass into positive commutator work.  The defect-work ratio \eqref{eq:comm-efficiency} formulates this as a scale-by-scale test.  Decay of this ratio would give a rigidity alternative; a persistent positive lower bound would isolate recurrent positive defect work encoded by the increment profile.

The estimates are structural. The singular near-field component of $\mathbb S_\ell\Omega_\ell\cdot\Omega_\ell$ is absorbed by diffusion, the far-field component is reduced to annular packing with explicit conditional closure requirements, and the commutator term is represented by a derivative-compatible increment defect together with a cylindrical obstruction profile.  Each surviving term is named, normalized, and paired with a specific summability, cancellation, or rigidity condition.

\section*{Declaration on AI assistance}
AI tools were used for language editing, organizational assistance, and preliminary reference-format checking.  All mathematical arguments, statements, proofs, and citations were reviewed and verified by the author, who takes full responsibility for the manuscript.

\end{document}